%% file: main.tex
\documentclass{wsbart}

\usepackage{lmodern}
\usepackage[T1]{fontenc}

\usepackage[maxnames=99,style=trad-plain,backend=biber,block=space]{biblatex}
\DeclareRedundantLanguages{English}{english}

\usepackage{amsmath, amssymb, amsthm}
\usepackage{mathtools}
\usepackage{dsfont}
\renewcommand{\mathbb}{\mathds}

\usepackage{graphicx}

\usepackage[ruled,vlined]{algorithm2e}

\theoremstyle{plain}
\newtheorem{thm}{Theorem}
\newtheorem{lem}[thm]{Lemma}
\newtheorem{cor}[thm]{Corollary}
\newtheorem{prop}[thm]{Proposition}

\theoremstyle{remark}
\newtheorem{rem}{Remark}
\newtheorem{exm}{Example}

\theoremstyle{definition}
\newtheorem{defn}{Definition}
\newtheorem{assu}{Assumption}
\newtheorem*{assu*}{Assumption}

\DeclareMathOperator{\Expect}{\mathbb E}

\DeclareMathOperator{\Law}{Law}
\DeclareMathOperator{\Var}{Var}
\DeclareMathOperator{\refc}{rc}
\DeclareMathOperator{\sync}{sc}
\DeclareMathOperator{\Id}{Id}
\newcommand{\dd}{\mathop{}\!\mathrm{d}}
\newcommand{\1}{\mathbb{1}}
\newcommand{\din}{d_\textnormal{in}}

\title{Self-interacting approximation\\
to McKean--Vlasov long-time limit:\\
a Markov chain Monte Carlo method}

\author[1,2]{Kai Du}
\author[3]{Zhenjie Ren}
\author[4]{Florin Suciu}
\author[5]{Songbo Wang}
\affil[1]{Shanghai Center for Mathematical Sciences,
Fudan University, Shanghai, China}
\affil[2]{Shanghai Artificial Intelligence Laboratory, Shanghai, China}
\affil[3]{LaMME, Université d'Évry, Université Paris-Saclay, Évry, France}
\affil[4]{CEREMADE, Université Paris-Dauphine, Université PSL, Paris, France}
\affil[5]{CMAP, École polytechnique, IP Paris, Palaiseau, France}

\begin{document}

\maketitle

\begin{abstract}
For a certain class of McKean--Vlasov processes,
we introduce proxy processes that substitute
the mean-field interaction with self-interaction,
employing a weighted occupation measure.
Our study encompasses two key achievements.
First, we demonstrate the ergodicity of the self-interacting dynamics,
under broad conditions, by applying the reflection coupling method.
Second, in scenarios where the drifts are negative intrinsic gradients
of convex mean-field potential functionals,
we use entropy and functional inequalities to demonstrate that
the stationary measures of the self-interacting processes approximate
the invariant measures of the corresponding McKean--Vlasov processes.
As an application, we show how to learn the optimal weights
of a two-layer neural network by training
a single neuron.
\end{abstract}

\input{simcmc.tex}
\appendix
\input{simcmc-app.tex}

\paragraph{Acknowledgements.}
The authors thank Fan Chen and Xiaozhen Wang for their assistance
with the numerical implementation,
and Yupeng Bai for identifying an error in Theorem~\ref{thm:si-contraction}
in a previous version of the paper.

\paragraph{Funding.}
The first author's research is supported
by the National Natural Science Foundation of China (No.~12222103),
and by the National Key R{\&}D Program of China (No.~2022ZD0116401).
The second author's research is supported by
Finance For Energy Market Research Centre.
The third author's research is supported by
the European Union’s Horizon 2020 research and innovation programme
under the Marie Skłodowska-Curie grant agreement No.~945332.

\printbibliography

\end{document}

%% file: simcmc.tex
\section{Introduction}

In this paper we develop a novel method to approximate the invariant measure
of the non-degenerate McKean--Vlasov dynamics
\begin{equation}
\label{eq:mkv}
\dd X_t = b\bigl(\Law(X_t), X_t\bigr) \dd t + \dd B_t,
\end{equation}
where $B$ is a standard Brownian motion in $\mathbb{R}^d$.
The McKean--Vlasov dynamics characterize
the mean field limit of interacting particles, and they
have widespread applications, encompassing fields such as
granular materials \cite{BCCPGranular, BGGGranular, CMVKinetic},
mathematical biology \cite{KellerSegelChemotaxis},
statistical mechanics \cite{MartzelAslangul}, and synchronization
of oscillators \cite{KuramotoOscillators}.
More recently, there has been a growing interest in the role of such dynamics
in the context of mean field optimization
for training neural networks
\cite{MMNMF,ChizatBachGlobal,HRSS,NWSPDA,%
NWSConvexMFL,ChizatMFL,CKRGame}.

In order to simulate the invariant measure of \eqref{eq:mkv},
we turn to the corresponding $N$-particle approximation, i.e., the dynamics
\begin{equation}
\label{eq:mkv-ps}
\dd X^i_t = b\biggl(\frac1N\sum_{i=1}^N \delta_{X^i_t}, X^i_t\biggr) \dd t
+ \dd B^i_t,
\qquad\text{for}~i = 1,\ldots,N,
\end{equation}
where $(B^i)_{1\leqslant i \leqslant N}$ are independent Brownian motions.
It is expected that the empirical measure $\frac1N\sum_{i=1}^N \delta_{X^i_t}$
of the $N$-particle system can approximate the McKean--Vlasov long-time limit
when $N$ and $t$ are both large enough.
For fixed $N$, to ensure control over the distance between
the McKean--Vlasov dynamics in \eqref{eq:mkv}
and the $N$-particle system in \eqref{eq:mkv-ps},
throughout the entire time horizon,
the literature has proposed various uniform-in-time propagation of chaos
results under different scenarios, see for example,
\cite{GuillinMonmarcheKinetic, DEGZElementary,
DelarueTseUniformPOC, LLFSharp, ulpoc}.

When the drift $b$ does not depend on the marginal distribution, $\Law(X_t)$,
the diffusion process is Markovian. Under general conditions,
we can leverage Birkhoff's ergodic theorem \cite{BirkhoffErgodic}
to approximate its invariant measure using the occupation measure
\[\overline m_t \coloneqq \frac1t \int_0^t \delta_{X_s}\dd s,\]
as $t\rightarrow\infty$.
In scenarios where the drift takes the form of gradient $b(x) = -\nabla U(x)$,
this ergodic property of the Markov diffusion lays the groundwork
for various Markov chain Monte Carlo methods,
such as Metropolis adjusted Langevin algorithm
\cite{RobertsTweedie,RobertsRosenthal}
and unadjusted Langevin algorithm \cite{DurmusMoulinesNonasymptotic}.
Motivated by the Markovian ergodicity, the recent paper \cite{DJLEmpApprox}
studied the following self-interacting process:
\begin{equation}
\label{eq:uniform-si}
\dd X_t = b(\overline m_t, X_t) \dd t + \dd B_t,
\end{equation}
where the dependence on the marginal distribution
in McKean-Vlasov diffusion \eqref{eq:mkv}
is replaced by the occupation measure,
that is, the empirical mean of the trajectory $(X_s)_{s \in [0,t]}$.
In \cite{DJLEmpApprox}, the authors successfully demonstrated that,
in the regime of weak interaction,
where the dependence on the marginal distribution is sufficiently small,
the occupation measures $(\overline m_t)_{t\geqslant 0}$
of the self-interacting process \eqref{eq:uniform-si}
also converge towards the invariant measure of \eqref{eq:mkv}
as $t\rightarrow\infty$.
Remarkably, from a practical point of view,
simulating the occupation measure of the self-interacting process
\eqref{eq:uniform-si} only requires a single particle,
which distinguishes it from the conventional $N$-particle approximation
\eqref{eq:mkv-ps}.
It is worth noting that the investigation into the long-time behavior
of the self-interacting diffusions can be traced back
to the pioneering works of Cranston and Le Jan
\cite{CranstonLeJanSIDiffusions}
and Raimond \cite{RaimondSIDiffusions}.

Building upon the discovery in \cite{DJLEmpApprox},
this paper ventures into uncharted territory,
where the mean field interaction need not to be inherently weak.
We propose to study the self-interacting particle
with exponentially decaying dependence on its trajectory:
\begin{equation}
\label{eq:si}
\begin{aligned}
\dd X_t &= b ( m_t, X_t ) \dd t + \dd B_t, \\
\dd m_t &= \lambda(\delta_{X_t} - m_t) \dd t,
\end{aligned}
\end{equation}
where $\lambda$ is a positive constant.
Integrating the second equation of \eqref{eq:si}, we find
\[
m_t = e^{-\lambda t} m_0
+ \int_0^t \lambda e^{-\lambda (t-s)} \delta_{X_s} \dd s,
\]
that is to say, the measure $m_t$ is
an exponentially weighted occupation measure with emphasis on the recent past.
The dynamics \eqref{eq:si} is a time-homogeneous Markov process
and we show its exponential ergodicity in the first part of the paper.
Although the state space where the random variable $(X_t, m_t)$ lives
is infinite-dimensional,
we have a non-degenerate noise in the $X$ component
and an almost sure contraction in the $m$ component,
which render such ergodicity possible.
Under suitable conditions for the drift $b$,
we show in Theorem~\ref{thm:si-contraction}
by a \emph{reflection coupling} approach that
the Markov process is contractive for a Wasserstein distance.
This implies that the stationary measure $\rho^\lambda$
for the Markov process exists, and is unique and globally attractive.
Notably, the conditions that we impose on $b$ do not imply
the uniqueness of invariant measure for the McKean--Vlasov
\eqref{eq:mkv},
and cover the case of the ferromagnetic Curie--Weiss model
in Section~\ref{sec:curie-weiss}.
Here, we also remark that
the exact weight for the measure $m_t$ is not important
and we work with a time-homogeneous Markov structure
only for convenience.
We could, for example, take the alternative weighting
\[
\frac{\lambda}{1 - e^{-\lambda t}} \int_0^t e^{-\lambda(t-s)}\delta_{X_s}\dd s
\]
to remove the dependency on the initial value.

We then proceed to investigate properties of the stationary $\rho^\lambda$
in Theorem~\ref{thm:dist-si-stationary-mkv-invariant}.
A crucial feature of the self-interacting process \eqref{eq:si}
is that when $\lambda \to 0$, the dynamics of the measure $m_t$ becomes slow,
while the rate of the $X_t$ dynamics remains roughly unchanged.
Suppose additionally that for some
mean field energy functional $F : \mathcal P(\mathbb R^d) \to \mathbb R$,
the drift is its negative intrinsic gradient:
\[
b(m,x) = -D_m F(m,x) = - \nabla_x \frac{\delta F}{\delta m}(m,x).
\]
The double time-scale structure allows us to speculate that
$\Law(X_t)$ rapidly relaxes to the local-in-time equilibrium
\[
\hat m_t \coloneqq
\frac{\exp \bigl( - 2 \frac{\delta F}{\delta m}(m_t,x)\bigr)\dd x}
{\int\exp \bigl( - 2 \frac{\delta F}{\delta m}(m_t,x')\bigr)\dd x'}
\]
so that the measure evolution is effectively described by
\begin{equation}
\label{eq:efp}
\partial_t m_t = \lambda (\hat m_t - m_t).
\end{equation}
This dynamics is called \emph{entropic fictitious play}
in a previous work of F.~Chen and two of the authors \cite{efp}
and this point of view plays a key role in various literatures,
notably the series of works of Benaïm, Ledoux and Raimond
\cite{BLRSIDiffusions,BRSIDiffusions2,BRSIDiffusions3,BRSIDiffusions4}
and the article of Kleptsyn and Kurtzmann \cite{KKErgodicitySI}.
The main novelty of our method is that we provide a \emph{quantitative}
justification of the intuition presented above,
and we are no longer restricted to the case of two-body interaction potentials.
More precisely, letting $(X,m)$ be a random variable
distributed as the stationary measure $\rho^\lambda$,
we provide an entropy bound
in Proposition~\ref{prop:rho-hat-rho-entropy}
that measures in a way the distance between
the conditional distribution $\Law(X|m)$ and $\hat m$,
relying crucially on the log-Sobolev inequality for $\hat m$.
This method requires unfortunately a finite-dimensional dependency
of the mean field in the energy functional $F$,
which we explain in Remark~\ref{rem:finite-dim} in detail.
The entropy bound,
together with an inherent gradient structure of the dynamics,
is then used in the rest of the proof of
Theorem~\ref{thm:dist-si-stationary-mkv-invariant}
to show that the random measure $m$ solves approximately
the stationary condition for the entropic fictitious play \eqref{eq:efp}:
\[
\hat m = m.
\]
In the case that the energy $F$ is convex,
the equation above has a unique solution $m_*$, which is also
the invariant measure for the McKean--Vlasov dynamics \eqref{eq:mkv},
and we show that the stationary measure $\rho^\lambda$
is in fact close to $m_* \otimes \delta_{m_*}$ for small $\lambda$.

The self-interacting dynamics \eqref{eq:si}
can be also thought as an intermediate scheme
between the entropic fictitious play \eqref{eq:efp},
which corresponds to the limit $\lambda \to 0$,
and the linear dynamics
\[
\dd X_t = b(\delta_{X_t}, X_t) \dd t + \dd B_t,
\]
which corresponds to the limit $\lambda \to \infty$.
From a computational point of view,
the linear dynamics is easy to sample and relaxes rapidly,
but in the long time does not yield the McKean--Vlasov's long-time limit.
The entropic fictitious play reaches high precision in the long time,
but at each time step, it requires usually a costly Monte Carlo run
to sample the Gibbs measure $\hat m_t$.
The self-interacting dynamics lies exactly in between
by mixing the two time scales.

As a final note on the terminology,
although the words ``stationary'' and ``invariant'' have almost identical
meanings in the context of stochastic process,
we always say ``invariant measure'' when referring to the McKean--Vlasov
process \eqref{eq:mkv},
and ``stationary measure'' when referring to the self-interacting process
\eqref{eq:si}.
We hope this artificial distinction would reduce possible confusions
for the readers.

\bigskip

The rest of the paper is organized as follows.
The main results are stated in Section~\ref{sec:main-results}.
Before moving to their proofs, we apply our results
to the training of a two-layer neural network in Section~\ref{sec:numerical}
and to a Curie--Weiss model for ferromagnets in Section~\ref{sec:curie-weiss}.
Ergodicity of the self-interacting dynamics,
i.e., Theorem~\ref{thm:si-contraction},
is proved in Section~\ref{sec:si-contraction}.
In Section~\ref{sec:gd},
we prove Theorem~\ref{thm:dist-si-stationary-mkv-invariant},
which characterizes the stationary measure
of the self-interacting process.
Finally, a technical result and its proof, and the numerical algorithm
are included in the appendices.

\section{Main results}
\label{sec:main-results}

We state and discuss our main results in this section.
First, we study the contractivity of the self-interacting process \eqref{eq:si},
and in particular, the exponential convergence
to its unique stationary measure is obtained.
Then, focusing on the gradient case,
we quantify the distance between the self-interacting stationary measure
and the corresponding McKean--Vlasov invariant measure.
Finally, applying both the results, we propose an annealing scheme
so that the self-interacting dynamics
converges to the McKean--Vlasov invariant measure.

To avoid extra assumptions on the drift functional $b$,
we will always assume the existence and the uniqueness of strong solution
to the self-interacting process \eqref{eq:si}
without explicitly mentioning it in the rest of the paper.

\begin{assu*}
Given any filtered probability space
supporting a Brownian motion $(B_t)_{t \geqslant 0}$
and satisfying the usual conditions,
for any initial conditions $(X_0, m_0)$
measurable to the initial $\sigma$-algebra
and taking value in $\mathbb R^d \times \mathcal P(\mathbb R^d)$,
there exists a unique adapted stochastic process $(X_t, m_t)_{t \geqslant 0}$
such that for all $t \geqslant 0$,
\begin{align*}
X_t &= \int_0^t b(m_s, X_s) \dd s + B_t + X_0, \\
m_t &= \lambda \int_0^t (\delta_{X_s} - m_s) \dd s + m_0.
\end{align*}
\end{assu*}

One may easily find various sufficient conditions for the assumption above.
For example, if
$b : \mathcal P(\mathbb R^d) \times \mathbb R^d \to \mathbb R^d$
is $W_1$-Lipschitz continuous in measure and Lipschitz continuous in space,
then by Cauchy--Lipschitz arguments, we know that the assumption is satisfied.

\subsection{Contractivity of the self-interacting diffusion}

We first present the results on the contractivity of
the self-interacting dynamics \eqref{eq:si},
from which follows the convergence to its unique stationary measure.
We start with two basic definitions. First, we define moduli of continuity.

\begin{defn}[Modulus of continuity]
\label{defn:moc}
We say that $\omega : [0,\infty) \to [0,\infty)$
is a modulus of continuity if it satisfies the following properties:
\begin{itemize}
\item $\omega(0) = 0$;
\item $\omega$ is continuous and non-decreasing;
\item for all $h$, $h' \geqslant 0$, we have
$\omega(h+h') \leqslant \omega(h) + \omega(h')$.
\end{itemize}
\end{defn}

Note that a modulus of continuity necessarily has at most linear growth
according to our definition.
We also introduce the notion of semi-monotonicity
following Eberle \cite{EberleReflectionCoupling}.

\begin{defn}[Semi-monotonicity]
\label{defn:semi-monotone}
We say that $\kappa : (0, \infty) \to \mathbb R$
is a semi-monotonicity function
for a vector field $v : \mathbb R^d \to \mathbb R^d$ if
\[
\bigl( v(x) - v(x') \bigr) \cdot (x - x')
\leqslant - \kappa(|x - x'|) |x - x'|^2
\]
holds for every $x$, $x' \in \mathbb R^d$ with $x' \neq x$.
We say $\kappa$ is a \emph{uniform} semi-monotonicity function
of a family of vector fields if it is a semi-monotonicity function
of each member.
\end{defn}

In this subsection, we impose the following assumption on the drift
of the McKean--Vlasov dynamics \eqref{eq:mkv}.

\begin{assu}
\label{assu:si-contraction}
The drift $b$ satisfies the following conditions:
\begin{enumerate}
\item For any fixed $m \in \mathcal P(\mathbb R^d)$,
the vector field $x \mapsto b(m,x)$ is uniformly equicontinuous
and has a uniform semi-monotonicity function $\kappa_b$, given by
$\kappa_b (x) = \kappa_0 - M_b/x$
for some $\kappa_0 > 0$ and $M_b \geqslant 0$.
\item There exist a bounded modulus of continuity
$\omega : [0,\infty)\to[0, M_\omega]$
and a constant $L \geqslant 0$ such that
\[
|b(m,x) - b(m',x)| \leqslant LW_\omega(m,m')
\]
for every $m$, $m' \in \mathcal P(\mathbb R^d)$
and every $x \in \mathbb R^d$.
Here $W_\omega$ is the Wasserstein distance
\[
W_\omega(m,m') = \inf_{\pi \in \Pi(m, m')} \int \omega(\lvert x - x'\rvert)
\pi (\dd x\dd x').
\]
\end{enumerate}
\end{assu}

Using reflection coupling, we derive the following result.
\begin{thm}
\label{thm:si-contraction}
Suppose Assumption~\ref{assu:si-contraction} hold.
Let $(X_t, m_t)_{t \geqslant 0}$, $(X'_t, m'_t)_{t \geqslant 0}$
be two processes following the dynamics \eqref{eq:si} for some $\lambda > 0$
such that the first marginals of their initial values $X_0$, $X'_0$
have finite first moments.
Define the following metric on $\mathbb R^d \times \mathcal P(\mathbb R^d)$:
\[
d_{\lambda} \bigl( (x,m), (x',m') \bigr)
= \lvert x - x' \rvert + \frac{2L}{\lambda} W_\omega(m, m')
\]
and denote the corresponding Wasserstein distance on
$\mathcal P_1\bigl(\mathbb R^d \times \mathcal P(\mathbb R^d)\bigr)$
by $W_{d_{\lambda}}$.
Then, we have
\[
W_{d_{\lambda}} \bigl( \Law(X_t, m_t), \Law(X'_t, m'_t) \bigr)
\leqslant
C e^{-ct} W_{d_{\lambda}} \bigl( \Law(X_0, m_0), \Law(X'_0, m'_0) \bigr),
\]
where the constants $C$, $c$ are given by
\begin{align*}
C &= 1 + \frac{2M}{\sqrt{K_0}} \exp \biggl( \frac{M^2}{4K_0} \biggr), \\
c &= \Biggl( \frac 1{K_0} + \frac{2M}{K_0^{3/2}}
\exp \biggl( \frac{M^2}{4K_0} \biggr) \Biggr)^{\!-1}
\end{align*}
for $M = M_b + 2L M_\omega$
and $K_0 = \min \bigl(\kappa_0, \frac{\lambda}{2}\bigr)$.
\end{thm}

The proof of Theorem~\ref{thm:si-contraction} is postponed
to Section~\ref{sec:si-contraction}.

\begin{rem}[On the assumption]
\label{rem:assum-si-contraction}
The first condition on the semi-monotonicity
of the vector field $x \mapsto b(m,x)$
is stronger than those used in \cite{EberleReflectionCoupling},
in that we require a more gentle singularity in $\kappa(x)$
for $x$ close to $0$.
This is because in this work, we are concerned with a good convergence rate
when the parameter $\lambda \to 0$
(see the following remark for more discussions)
and it will become clear in the proof that this stronger requirement
on the semi-monotonicity is necessary for our purpose.
Nevertheless, this condition is not too difficult to fulfill.
Upon defining $b_0(x) = - \kappa_0 x$ and $b_1(m,x) = b(m,x) - b_0(x)$,
the first condition of Assumption~\ref{assu:si-contraction} is equivalent to
\[
\bigl( b_1(m,x) - b_1(m,x'), x - x'\bigr) \leqslant M_b,
\]
and this holds true
when $\sup_{(m,x)\in\mathcal P(\mathbb R^d)\times\mathbb R^d}
\lvert b_1(m,x) \rvert \leqslant M_b / 2$ in particular.
\end{rem}

\begin{rem}[Rate of convergence]
\label{rem:si-contraction-rate}
We discuss the rate of convergence across three parameter ranges.

In the first scenario, when the drift $b$ is $\kappa_0$-strongly monotone
for some $\kappa_0 > 0$, i.e., $M_b = 0$,
and when there is no mean field interaction ($L = 0$),
we have $M = 0$ and $K_0 = \min\bigl(\kappa_0, \frac{\lambda}{2}\bigr)$.
Consequently, $C = 1$ and $c = \min\bigl(\kappa_0, \frac{\lambda}{2}\bigr)$.
It is worth noting that under these conditions,
the component $X$ exhibits exponential contraction with a rate of $\kappa_0$,
while $m$ contracts at a rate no greater than $\lambda$.
In this case, the best contraction rate for the joint process
is $\min(\kappa_0, \lambda)$.
Thus our method yields a contraction rate that
remains at least half of the optimal one.

In the second scenario, when $\lambda$ is small but non-zero
(with self-interaction),
we obtain $c \sim 2M\lambda^{3/2}\exp(-M^2\!/2\lambda)$
and $C \sim 2M\lambda^{-1/2}\exp(-M^2\!/2\lambda)$.
We note that such exponentially slow convergence also arises
in the kinetic Langevin process as the damping parameter approaches zero;
see Eberle, Guillin and Zimmer \cite[Section~2.6]{EGZCoupling}
for further discussion.

Finally, for $\lambda > 2\kappa_0$, the contractivity constants $C$, $c$
become independent of $\lambda$, consistent with the intuition that
the self-interacting diffusion becomes linear in the large $\lambda$ limit.
\end{rem}

Now we discuss a few examples satisfying Assumption~\ref{assu:si-contraction}.

\begin{exm}[Two-body interaction]
\label{exm:two-body}
Consider $b (m,x) = b_0(x) + b_1(m,x)$, where
\[
b_1 (m,x) = \int K (x,x') m(\dd x').
\]
Suppose furthermore that
\[
\sup_{x\in\mathbb R^d}\lVert K(x,\cdot)\rVert_{W^{1,\infty}}
= \sup_{x\in\mathbb R^d}\max\bigl(\lVert K(x,\cdot)\rVert_{L^\infty},
\lVert \nabla K(x,\cdot)\rVert_{L^\infty}\bigr) \leqslant M,
\]
that is to say, the mapping $y \mapsto K(x,y)$ is $M$-bounded
and $M$-Lipschitz continuous for every $x$.
Thus, we have
\[
|b(m,x) - b(m',x)|
\leqslant \biggl| \int K(x, x') (m-m')(\dd x') \biggr|
\leqslant M W_{\omega} (m,m')
\]
for the modulus of continuity $\omega(x) = \min(x,2)$.
Therefore Assumption~\ref{assu:si-contraction}
is satisfied once $b_0$ is uniformly Lipschitz
and has a semi-monotonicity function
$\kappa_0(x) = \kappa_0 - M_1 / x$
for some $\kappa_0 > 0$ and $M_1 \geqslant 0$.
\end{exm}

We can generalize the example above to drifts that depend on the measure
in a non-linear way.

\begin{exm}[$\mathcal C^1$ functional]
\label{exm:c1}
Suppose $m \mapsto b(m,x)$ is $\mathcal C^1$ differentiable
in the sense that there exists a continuous and bounded mapping
$\frac{\delta b}{\delta m} : \mathcal P(\mathbb R^d)
\times \mathbb R^d \times \mathbb R^d \to \mathbb R$
such that
\[
b(m,x) - b(m',x) =
\int_0^1\!\!\int \frac{\delta b}{\delta m} \bigl((1-t)m+tm',x,x'\bigr)
(m-m')(\dd x') \dd t
\]
for all $m$, $m' \in \mathcal P(\mathbb R^d)$ and $x\in\mathbb R^d$.
If $\sup_{m,x} \bigl\Vert \frac{\delta b}{\delta m}(m,x,\cdot)
\bigr\Vert_{W^{1,\infty}}$
is finite and the vector fields $x \mapsto b(m,x)$
are uniformly Lipschitz and share a uniform semi-monotonicity function
$\kappa_0(x) = \kappa_0 - M_1 / x$
for some $\kappa_0 > 0$ and $M_1 \geqslant 0$,
then by the same argument as in Example~\ref{exm:two-body},
Assumption~\ref{assu:si-contraction} is satisfied.
\end{exm}

We now examine the stationary measure of the self-interacting process
\eqref{eq:si}.

\begin{defn}
\label{defn:si-stationary}
We call a probability measure
$P \in \mathcal{P}\bigl(\mathbb{R}^d \times \mathcal{P}(\mathbb{R}^d)\bigr)$
stationary to the self-interacting diffusion \eqref{eq:si}
if the stochastic process $(X_t, m_t)_{t \geqslant 0}$ with initial value
$\Law(X_0, m_0) = P$ satisfies $\Law(X_t, m_t) = P$ for all $t \geqslant 0$.
\end{defn}

The definition above makes sense since we have assumed
the existence and uniqueness of strong solution.

By the Banach fixed point theorem in the metric space
$\mathcal P_1\bigl(\mathbb R^d \times \mathcal P(\mathbb R^d)\bigr)$
and standard arguments,
Theorem~\ref{thm:si-contraction} implies
the existence and uniqueness
of stationary measure of the self-interacting process \eqref{eq:si}.

\begin{cor}
\label{cor:si-stationary-exist-unique}
Under Assumption~\ref{assu:si-contraction},
for every $\lambda > 0$,
there exists a unique stationary measure
of the self-interacting diffusion \eqref{eq:si}
in $\mathcal P\bigl(\mathbb R^d \times \mathcal P(\mathbb R^d)\bigr)$
whose first marginal distribution on $\mathbb{R}^d$ has finite first moment.
\end{cor}

Finally, we note that although Theorem~\ref{thm:si-contraction},
along with Corollary~\ref{cor:si-stationary-exist-unique},
implies that the self-interacting process \eqref{eq:si}
converges to its stationary measure exponentially,
its assumptions are not sufficient to establish
the uniqueness of invariant measure
of the McKean--Vlasov process \eqref{eq:mkv},
as illustrated by the Curie--Weiss example in Section~\ref{sec:curie-weiss}.
So in general, there is no hope that
the self-interacting stationary measure approximates
the McKean--Vlasov invariant measure.

\subsection{Stationary measure in the gradient case}\label{subsec:gradientcase}

In this subsection,
we study the stationary measure of the self-interacting dynamics \eqref{eq:si},
provided that the drift $b$ is the negative intrinsic gradient
of a finite-dimensional mean field functional,
whose precise meaning will be explained in the following.
In particular, we aim at proving that, in this case,
the stationary measure of \eqref{eq:si} provides an approximation to
the invariant measure of the McKean--Vlasov dynamics \eqref{eq:mkv}.
We fix a positive $\lambda$ in this subsection.

The first assumption in the subsection is that the drift $b$ corresponds to
a gradient descent whose dependency in the mean field is finite-dimensional.

\begin{assu}[Finite-dimensional mean field]
\label{assu:finite-dim-gd}
For a closed convex set $\mathcal K \subset \mathbb R^D$,
there exists a function
\[
\ell = (\ell^1,\ldots, \ell^D) \in \mathcal C^1(\mathbb R^d; \mathcal K)
\]
whose gradient is of at most linear growth,
and a function $\Phi \in \mathcal C^2(\mathbb R^{D}; \mathbb R)$
whose Hessian $\nabla^2\Phi$ is bounded, such that the drift term $b$ reads
\begin{align*}
b(m,x)
&= - \nabla \Phi(\langle \ell, m\rangle) \cdot \nabla \ell(x)
= - \nabla \Phi \biggl( \int \ell (x) m(\dd x) \biggr) \cdot \nabla \ell (x) \\
&= - \sum_{\nu = 1}^D \nabla_\nu \Phi \biggl( \int \ell (x) m(\dd x) \biggr)
\nabla \ell^\nu (x).
\end{align*}
In other words, for the mean field functional
$F : \mathcal P_2(\mathbb R^d) \to \mathbb R$ defined by
\[
F(m) = \Phi (\langle \ell, m\rangle)
= \Phi \biggl( \int \ell(x') m(\dd x') \biggr),
\]
the drift $b$ is the negative intrinsic derivative:
\[
b(m,x) = - D_m F(m,x) = - \nabla_x \frac{\delta F}{\delta m}(m,x).
\]
\end{assu}

We shall also impose the following conditions
on a family of probability measures related to $\Phi$ and $\ell$.

\begin{assu}[Uniform LSI]
\label{assu:m-hat-uniform-lsi}
The probability measures $(\hat m_{y})_{y \in \mathcal K}$
on $\mathbb R^d$ determined by
\[
\hat m_y(\dd x)
\propto \exp \bigl( - 2 \nabla \Phi(y) \cdot \ell(x) \bigr) \dd x
\]
are well defined and satisfy a
uniform $C_\textnormal{LS}$-logarithmic Sobolev inequality
for some $C_\textnormal{LS} \geqslant 0$.
That is to say, for all regular enough probability measure
$m \in \mathcal P(\mathbb R^d)$ and all $y \in \mathcal K$, we have
\[
\int \log \frac{\dd m}{\dd \hat m_y} \dd m
\eqqcolon H(m | \hat m_y)
\leqslant \frac{C_\textnormal{LS}}{4} I(m | \hat m_y)
\coloneqq \frac{C_\textnormal{LS}}{4}
\int \biggl| \nabla \log \frac{\dd m}{\dd \hat m_y} \biggr|^2 \hat m_y,
\]
where $\dd m/\!\dd \hat m_y$ is the Radon--Nikod\'ym derivative
between measures.
\end{assu}

\begin{rem}
\label{rem:lsi}
As mentioned in the introduction,
the uniform log-Sobolev inequality is crucial to our method
as it is used to obtain the entropy estimate
in Proposition~\ref{prop:rho-hat-rho-entropy}.
This condition is often perceived as a strong one,
but still it can be verified if for example
\[
2 \nabla\Phi (y) \cdot \ell(x)
= U(x) + G(y,x)
\]
for some strongly convex $U : \mathbb R^d \to \mathbb R$
and some bounded $G : \mathcal K \times \mathbb R \to \mathbb R$.
Indeed, in this case $\hat m_y$ is a uniformly bounded perturbation
of the strongly log-concave measure
$e^{-U(x)}\dd x\big/\!\int e^{-U(x')}\dd x'$,
so it satisfies a uniform log-Sobolev inequality
by the Bakry--Émery condition \cite{BakryEmeryHypercontractives}
and the Holley--Stroock perturbation \cite{HolleyStroockLSI}.
We note that this condition has also been exploited
recently to obtain long-time behaviors
of mean field Langevin \cite{ChizatMFL,NWSConvexMFL}
and its particle systems \cite{ulpoc,uklpoc}.
\end{rem}

Finally, we assume that the following quantitative bound
on $\Phi$ and $\ell$.

\begin{assu}[Bound]
\label{assu:M2}
The following quantity is finite:
\begin{align*}
M_2 &\coloneqq \sup_{x \in \mathbb R^d, y \in \mathbb R^{D}}
\bigl\lvert \nabla^2 \Phi(y)^{1/2} \nabla \ell(x) \bigr\rvert^2 \\
&= \sup_{x \in \mathbb R^d, y \in \mathbb R^{D}}
\sup_{\substack{a \in \mathbb R^d\\ \lvert a\rvert =1}}
a^\top\nabla\ell(x)^\top \nabla^2 \Phi(y) \nabla \ell(x) a.
\end{align*}
\end{assu}

\begin{rem}
Under the three assumptions above, if $\Phi$ is additionally convex,
then there exists a unique invariant measure
$m_*$ of the McKean--Vlasov dynamics \eqref{eq:mkv}
by \cite[Proposition~4.4 and Corollary 4.8]{ulpoc},
and the convergence to the invariant measure is exponentially fast
\cite[Theorem 2.1]{ulpoc}.
In fact, the convexity of $\Phi$ implies precisely the functional convexity
of the mean field energy $F$ considered in \cite{ulpoc}.
\end{rem}

The main discovery of this paper characterizes the distance between
the stationary measure $P$ of the self-interacting dynamics \eqref{eq:si}
and $m_* \otimes \delta_{m_*}$.

\begin{thm}
\label{thm:dist-si-stationary-mkv-invariant}
Let Assumptions~\ref{assu:finite-dim-gd},
\ref{assu:m-hat-uniform-lsi} and \ref{assu:M2} hold true.
Suppose that $P = P^\lambda \in \mathcal P_4\bigl(
\mathbb R^d \times \mathcal P_4(\mathbb R^d)\bigr)$
is a stationary measure of the self-interacting process \eqref{eq:si}
in the sense of Definition~\ref{defn:si-stationary}
that has finite fourth moment:
\[
\iint \biggl( |x|^4 + \int |x'|^4 m(\dd x') \biggr) P(\dd x\dd m) < \infty.
\]
Let $(X, m)$ be a random variable distributed as $P$.
Denote by $\rho^1$ and $\rho^2$ the probability distributions
of the random variables $X$ and $\langle \ell, m\rangle = \int \ell(x) m(\dd x)$
respectively.

\begin{enumerate}
\item
Suppose $\Phi$ is convex.
Denote in the case by $m_*$ the unique invariant measure of the McKean--Vlasov
dynamics \eqref{eq:mkv}.
Define
\begin{align*}
y_* &\coloneqq \langle\ell, m_*\rangle = \int \ell(x)m_*(\dd x),\\
\lambda_0 &\coloneqq
\frac{1}{48M_2 C_\textnormal{LS}^2
\Bigl( 1 + 2 M_2C_\textnormal{LS}
\bigl( M_2^2C_\textnormal{LS}^2+1\bigr)^{1/2} \Bigr)}, \\
H &\coloneqq
\frac{C_\textnormal{LS}(D + 24 M_2C_\textnormal{LS}d)\lambda}
{2 - 96 M_2 C_\textnormal{LS}^2
\Bigl( 1 + 2 M_2C_\textnormal{LS}
\bigl( M_2^2C_\textnormal{LS}^2+1\bigr)^{1/2} \Bigr) \lambda}.
\end{align*}
Then, for
$\lambda \in (0, \lambda_0)$, we have
\begin{align*}
v(\rho^2)&\coloneqq \Expect \biggl[
\iiint_0^1
\frac{\delta^2 F}{\delta m^2}
\bigl((1-t)m + tm_*, x', x''\bigr)
\dd t\,(m - m_*)^{\otimes 2}(\dd x'\dd x'')
\biggr] \\
&= \iint_0^1
(y - y_*)^\top\nabla^2\Phi\bigl((1-t)y+ty_*\bigr)(y-y_*)
\dd t\, \rho^2(\dd y) \\
&\leqslant 4M_2C_\textnormal{LS}\bigl(M_2^2C_\textnormal{LS}^2+1\bigr)H,
\end{align*}
and
\begin{align*}
W_2^2 (\rho^1,m_*) &\leqslant
\Bigl( 2C_\textnormal{LS} + 4M_2C_\textnormal{LS}^2
\bigl(M_2^2C_\textnormal{LS}^2+1\bigr)^{1/2}\Bigr) H, \\
\lVert \rho^1 - m_*\rVert_\textnormal{TV}^2 &\leqslant
\Bigl( 4 + 8M_2C_\textnormal{LS}
\bigl(M_2^2C_\textnormal{LS}^2+1\bigr)^{1/2}\Bigr) H.
\end{align*}

\item
If in addition to the convexity of $\Phi$,
\[
M_1 \coloneqq \sup_{x \in \mathbb R^d, y \in \mathbb R^D}
\ell(x)^\top \nabla^2\Phi(y) \ell(x) < \infty,
\]
then the three inequalities above hold also
for all $\lambda \in (0, \infty)$, with $H$ replaced by
\[
H' \coloneqq \frac{C_\textnormal{LS}}{2}(D + 2M_1) \lambda.
\]

\item
If $\Phi$ is concave, then
for $\hat y \coloneqq \langle \ell, \hat m_y \rangle$,
we have
\[
-\int (\hat y-y)^\top \nabla^2\Phi(y) (\hat y-y) \rho^2(\dd y)
\leqslant \frac{M_2C_\textnormal{LS}^2D}{2} \lambda.
\]
\end{enumerate}
\end{thm}

The proof of Theorem~\ref{thm:dist-si-stationary-mkv-invariant} is postponed to
Section~\ref{sec:gd}.

\begin{rem}[Dependence on the dimension $D$]
\label{rem:finite-dim}
Readers may have observed that, in our framework,
the value of the functional $F(m)$ may only depend
on the $D$-dimensional vector $\int \ell(x) m(\dd x)$,
and this corresponds to ``cylindrical functions'' considered
in \cite[Definition 5.1.11]{AGSGradientFlows}.
Given a continuous functional $F$ on $\mathcal P(\mathbb R^d)$,
for every compact subset $\mathcal S\subset\mathbb{R}^d$,
we can construct, according to
the Stone--Weierstrass theorem, a sequence
of functions $\ell_n : \mathbb R^d \to \mathbb R^{D_n}$
and $\Phi_n : \mathbb R^{D_n} \to \mathbb R$
such that the cylindrical functionals
$F_n(m) = \Phi_n \bigl( \int \ell_n \dd m \bigr)$ approach $F$
in the uniform topology of $\mathcal C\bigl(\mathcal P(\mathcal S)\bigr)$
(see \cite[Lemma 2]{FlemingViot} for example).
However, the dimension $D_n$ may tend to infinity when $n \to \infty$.
Since all the upper bounds
in Theorem~\ref{thm:dist-si-stationary-mkv-invariant}
depend on the dimension $D$ linearly, our analysis and findings cannot
be directly applied to more general functionals on $\mathcal P(\mathbb R^d)$.
\end{rem}

\smallbreak

\begin{rem}[Meaning of $M_1$ and $M_2$]
The second-order functional derivative of $F$ reads
\[
\frac{\delta^2F}{\delta m^2}(m,x',x'')
= \ell(x')^\top \nabla^2\Phi(y) \ell(x''),
\]
and in the case of convex $\Phi$,
satisfies the Cauchy--Schwarz inequality:
\begin{align*}
\biggl| \frac{\delta^2F}{\delta m^2}(m,x',x'') \biggr|
&= \bigl| \ell(x')^\top \nabla^2\Phi(y) \ell(x'') \big| \\
&\leqslant \bigl| \ell(x')^\top \nabla^2\Phi(y) \ell(x') \big|^{1/2}
\bigl| \ell(x'')^\top \nabla^2\Phi(y) \ell(x'') \big|^{1/2} \\
&= \biggl| \frac{\delta^2F}{\delta m^2}(m,x',x') \biggr|^{1/2}
\biggl| \frac{\delta^2F}{\delta m^2}(m,x'',x'') \biggr|^{1/2} \leqslant M_1.
\end{align*}
Similarly, the second-order intrinsic derivative satisfies
\begin{align*}
\bigl|D_m^2 F(m,x',x'')\bigr|
&= \bigl|\nabla\ell(x')^\top \nabla^2\Phi(y) \nabla\ell(x'')\bigr| \\
&\leqslant \bigl|\nabla\ell(x')^\top \nabla^2\Phi(y) \nabla\ell(x')\bigr|^{1/2}
\bigl|\nabla\ell(x'')^\top \nabla^2\Phi(y) \nabla\ell(x'')\bigr|^{1/2} \\
&= \bigl| D_m^2 F(m,x',x') \bigr|^{1/2}
\bigl| D_m^2 F(m,x'',x'') \bigr|^{1/2} \leqslant M_2.
\end{align*}
Moreover, by taking $x' = x''$ in the inequalities above,
we find that $M_1$ and $M_2$ are the respective suprema
of the second-order flat and intrinsic derivatives of the functional $F$.
\end{rem}

\smallbreak

We illustrate in the following example that the order of $\lambda$
when $\lambda \to 0$ for the variance of $\langle \ell, m\rangle$ under $P$
in Theorem~\ref{thm:dist-si-stationary-mkv-invariant}
(i.e., the first claim) is optimal.

\smallbreak

\begin{exm}[Optimality of the order of $\lambda$]
\label{exm:optimal-order}
Consider the mean field functional
$F : \mathcal P_2(\mathbb R^d) \to \mathbb R$ given by
\[
F(m) = \frac 12 \int \lvert x\rvert^2 m(\dd x)
+ \frac 12 \biggl| \int x m(\dd x) \biggr|^2.
\]
By taking $D = d+1$,
$\ell(x) = (x,\lvert x\rvert^2\!/2)^\top$,
$\Phi(y_0, y_1) = y_0^2/2 + y_1$,
the mean field functional $F$ is covered by the cylindrical setting
(namely Assumption~\ref{assu:finite-dim-gd})
of Theorem~\ref{thm:dist-si-stationary-mkv-invariant}.
Moreover, the function $\Phi$ is convex.

The corresponding gradient dynamics \eqref{eq:mkv}
is then characterized by the drift
\[
b(m,x) = - x - \int x' m(\dd x'),
\]
and its unique invariant measure $m_*$ is $\mathcal N(0,1/2)$.
We explicitly compute quantities related to the stationary measure $P^\lambda$
of the self-interacting dynamics \eqref{eq:si} in the following.
The dynamics reads
\begin{align*}
\dd X_t &= - X_t \dd t - \int x' m_t(\dd x') \dd t + \dd B_t, \\
\dd m_t &= \lambda (\delta_{X_t} - m_t) \dd t,
\intertext{and has a unique strong solution by Cauchy--Lipschitz arguments.
Upon defining $Y_{0,t} = \int x' m_t(\dd x')$, the process
has the finite-dimensional projection:}
\dd X_t &= (-X_t - Y_{0,t}) \dd t + \dd B_t, \\
\dd Y_{0,t} &= \lambda (-Y_{0,t} + X_t) \dd t.
\end{align*}
The finite-dimensional dynamics has a unique invariant measure
that is a centered Gaussian with the following covariance structure:
\begin{align*}
\Expect [X\otimes X] &= \frac{\lambda+2}{4(\lambda+1)}
\1_{d\times d}, \\
\Expect[Y_0\otimes Y_0] = \Expect[X\otimes Y_0]
&= \frac{\lambda}{4(\lambda+1)} \1_{d\times d}.
\end{align*}
Hence, the exact distances, or bounds thereof, read
\begin{align*}
\Expect\bigl[\lvert Y_0\rvert^2\bigr]
&= \frac{d\lambda}{4(1+\lambda)}, \\
W_2^2\bigl(\Law(X), m_*\bigr) &=
\frac{d}{2} \Biggl( 1 - \biggl( 1 - \frac{\lambda}{2(1+\lambda)}\biggr)^{\!1/2}
\Biggr)^{\!2}, \\
\lVert \Law(X) - m_*\rVert_\textnormal{TV}^2
&\in \biggl[ \frac 1{10000} \frac{d\lambda^2}{4(1+\lambda)^2},\,
\frac 94 \frac{d\lambda^2}{4(1+\lambda)^2}\biggr],
\end{align*}
where the last mutual bound is due to
\cite[Theorem 1.1]{DMRTVHighDimGaussians}.

Now we try to verify the assumptions
of Theorem~\ref{thm:dist-si-stationary-mkv-invariant}.
By the Kolmogorov extension theorem, we can construct
a stationary Markov process $(X_t, Y_{0,t})_{t \in \mathbb R}$,
defined on the whole real line, such that
$\Law(X_t, Y_{0,t}) = \Law (X, Y_0)$ for all $t \in \mathbb R$.
Then, by defining
\[
m_t = \lambda \int_{- \infty}^t e^{-\lambda (t-s)} \delta_{X_s} \dd s,
\]
we recover the solution $(X_t, m_t)_{t \in \mathbb R}$
to the original infinite-dimensional dynamics.
By construction, $\Law(X_t, m_t)$ is stationary
and has finite fourth moments.
The rest of the assumptions
of Theorem~\ref{thm:dist-si-stationary-mkv-invariant}
can be satisfied with the constants
$C_\textnormal{LS} = 1/2$ and $M_2 = 1$.
Since $\Phi$ is convex,
by the first claim of the theorem, we get
\begin{align*}
\Expect\bigl[\lvert Y_0\rvert^2\bigr] &\leqslant
\frac{5(13d+1)\lambda}{8 - 48(2+\sqrt 5)\lambda}, \\
W_2^2\bigl(\Law(X), m_*\bigr) &\leqslant
\frac{(2+\sqrt 5)(13d+1)\lambda}{8 - 48(2+\sqrt 5)\lambda}, \\
\lVert \Law(X) - m_* \rVert_\textnormal{TV}^2 &\leqslant
\frac{(2+\sqrt 5)(13d+1)\lambda}{2 - 12(2+\sqrt 5)\lambda},
\end{align*}
for $\lambda < 1 \big/6(2+\sqrt 5)$.
So the results of Theorem~\ref{thm:dist-si-stationary-mkv-invariant}
give the optimal order of $\lambda$ when $\lambda \to 0$
for the variance of the $Y$ variable,
but possibly sub-optimal ones for the Wasserstein and total-variation distances
in the $X$ direction.
\end{exm}

\subsection{A class of dynamics}

In this subsection, we present a class of dynamics
to which both Theorems~\ref{thm:si-contraction}
and \ref{thm:dist-si-stationary-mkv-invariant} are applicable.
This class encompasses in particular the neural network example
that will be discussed in the following Section~\ref{sec:numerical}.

\begin{assu}
\label{assu:example-dynamics}
The drift functional writes
\begin{equation}
\label{eq:drift-example-dynamics}
b(m,x) = - \nabla U(x)
- \nabla \Phi_0 \biggl( \int \ell_0(x') m(\dd x') \biggr)
\cdot \nabla \ell_0 (x).
\end{equation}
for some functions $U : \mathbb R^d \to \mathbb R$,
$\Phi_0 : \mathbb R^D \to \mathbb R$,
$\ell_0 : \mathbb R^d \to \mathbb R^D$ satisfying the following conditions:
\begin{itemize}
\item the function $U$ is $\mathcal C^2$ continuous with bounded Hessian, i.e.,
$\lVert \nabla^2 U \rVert_{L^\infty} < \infty$, and its gradient admits a
semi-monotonicity function $\kappa(x) = \kappa_0 - M/x$ for some $\kappa_0 > 0$
and $M \geqslant 0$.
\item the probability measure $Z^{-1} \exp \bigl( - 2U(x) \bigr) \dd x$,
with $Z = \int \exp \bigl( - 2U(x) \bigr) \dd x$, is well defined
in $\mathcal P(\mathbb R^d)$,
and satisfies a $C_{\textnormal{LS}, 0}$-logarithmic Sobolev inequality.
\item the function $\Phi_0$ is $\kappa_{\Phi_0}$-strongly convex
for some $\kappa_{\Phi_0} > 0$
and belongs to $\mathcal C^2(\mathbb R^D) \cap W^{2,\infty}(\mathbb R^D)$.
\item the function $\ell_0$ belongs to $\mathcal C^1(\mathbb R^d; \mathbb R^D)
\cap W^{1,\infty} (\mathbb R^d; \mathbb R^D)$.
\end{itemize}
\end{assu}

\begin{prop}
\label{prop:example-dynamics}
Under Assumption~\ref{assu:example-dynamics},
there exists a unique stationary measure
$P \in \mathcal P_1\bigl(\mathbb R^d\times\mathcal P(\mathbb R^d)\bigr)$
to the self-interacting dynamics \eqref{eq:si}.
Moreover, there exists $C > 0$, independent of $\lambda$, such that
for $(X, m)$ distributed as $P$,
\[
\Expect[\lvert \langle \ell_0, m-m_*\rangle \rvert^2]
+ W_2^2\bigl(\Law(X),m_*\bigr) +
\lVert \Law(X) - m_* \rVert_\textnormal{TV}^2 \leqslant C\lambda,
\]
where $m_*$ is the unique invariant measure to the McKean--Vlasov process.
\end{prop}

\begin{proof}[Proof of Proposition~\ref{prop:example-dynamics}]
We first verify the conditions of Theorem~\ref{thm:si-contraction}
to establish the existence and uniqueness of the stationary measure.
As the drift $b$ has derivative
\[
\frac{\delta b}{\delta m}(m,x,x')
= - \nabla^2 \Phi_0 \biggl( \int \ell_0(x'') m(\dd x'') \biggr)
\cdot \nabla \ell_0(x)
\cdot \ell_0(x'),
\]
we have
\[
\biggl\Vert\frac{\delta b}{\delta m}(m,x,\cdot)\biggr\Vert_{W^{1,\infty}}
\leqslant \lVert \nabla^2 \Phi_0 \rVert_{L^\infty}
\lVert \nabla \ell_0 \rVert_{L^\infty}
\lVert \ell_0 \rVert_{W^{1,\infty}}
\]
for every $m \in \mathcal P(\mathbb R^d)$ and every $x \in \mathbb R^d$.
Then the dynamics falls into the class considered in Example~\ref{exm:c1}.
The conditions of Theorem~\ref{thm:si-contraction} are satisfied.
Applying Corollary~\ref{cor:si-stationary-exist-unique} gives
the existence and the uniqueness of the stationary measure
$P \in \mathcal P_1\bigl(\mathbb R^d \times \mathcal P(\mathbb R^d)\bigr)$.

We proceed to verify the conditions of
Theorem~\ref{thm:dist-si-stationary-mkv-invariant}.
Introduce the functions
$\Phi : \mathbb R^{D+1} \to \mathbb R$,
$\ell : \mathbb R^d \to \mathbb R^{D+1}$
defined respectively by
\begin{align*}
\Phi(y_0, y_1) &= \Phi_0 (y_0) + y_1,
&&\text{for}~y_0\in\mathbb R^D~\text{and}~y_1\in\mathbb R, \\
\ell(x) &= \bigl( \ell_0(x), U(x) \bigr),
&&\text{for}~x\in\mathbb R^d.
\end{align*}
Here the range set $\mathcal K$ of the mapping $\ell$
is taken as the whole space $\mathbb R^D$.
In this way, the drift $b$ reads
\[
b(m,x) = - \nabla \Phi\biggl( \int \ell(x') m(\dd x') \biggr)
\cdot \nabla \ell(x)
\]
so Assumption~\ref{assu:finite-dim-gd} is satisfied.
Now we show the stationary measure $P$ of the dynamics \eqref{eq:si}
has finite fourth moment.
Consider the functional
\[
E (x,m) = \kappa_0^{-1} \lvert x\rvert^4 + \int \lvert x'\rvert^4 m(\dd x').
\]
Along the self-interacting dynamics \eqref{eq:si},
we have, by It\=o's formula,
\begin{multline*}
\frac{\dd}{\dd t} \Expect [E(X_t, m_t)]
= 4 \kappa_0^{-1} \Expect \bigl[b(m_t, X_t) \cdot X_t \lvert X_t\rvert^2\bigr]
+ (2d + 4) \kappa_0^{-1}\Expect\bigl[\lvert X_t\rvert^2\bigr] \\
- \lambda \Expect \biggl[ \int \lvert x'\rvert^4 m'(\dd x') \biggr]
+ \lambda \Expect \bigl[ \lvert X_t \rvert^4 \bigr].
\end{multline*}
As the vector field $x \mapsto b(m,x)$ has weak monotonicity function
$\kappa_b(x) = \kappa_0 - M_b / x$, we have
\[
b(m,x) \cdot x \leqslant - \frac{\kappa_0}{2} \lvert x\rvert^2 + C_2
\]
for every $(m,x) \in \mathcal P(\mathbb R^d) \times \mathbb R^d$,
for some $C_2 \geqslant 0$.
The functional $E$ verifies the Lyapunov condition as well:
\[
\frac{\dd}{\dd t} \Expect [E(X_t, m_t)]
\leqslant - c_3 \Expect [E(X_t, m_t)] + C_3
\]
for some $c_3 > 0$, $C_3 \geqslant 0$.
In consequence, the stationary measure $P$
of finite first moment must have finite fourth moment.
Since
\[
\hat m_y(\dd x)
\propto \exp\bigl(- 2 \nabla \Phi (y) \cdot \ell(x)\bigr) \dd x
= \exp\bigl(- 2 \nabla \Phi_0 (y) \cdot \ell_0(x)\bigr)
\exp\bigl(-2 U(x)\bigr) \dd x
\]
with $\sup_{x,y}\lvert\nabla \Phi_0 (y) \cdot \ell_0(x)\rvert \leqslant
\lVert \nabla \Phi_0 \rVert_{L^\infty} \lVert \ell_0 \rVert_{L^\infty}$,
by the Holley--Stroock perturbation argument \cite{HolleyStroockLSI},
we know that the measure $\hat m_y$
verifies a uniform $C_\textnormal{LS}$-LSI with
\[
C_\textnormal{LS} = C_{\textnormal{LS},0}
\exp(4 \lVert \nabla \Phi_0 \rVert_{L^\infty} \lVert \ell_0 \rVert_{L^\infty}),
\]
so Assumption~\ref{assu:m-hat-uniform-lsi} is satisfied.
The constants $M_1$, $M_2$
in Theorem~\ref{thm:dist-si-stationary-mkv-invariant}
and Assumption~\ref{assu:M2} are finite as
\begin{align*}
\lvert\ell(x)^\top \nabla^2 \Phi(y) \ell(x)\rvert
&= \lvert\ell_0(x)^\top \nabla^2\Phi_0(y) \ell_0(x)\rvert
\leqslant \lVert\nabla^2\Phi_0\rVert_{L^\infty}
\lVert\ell_0\rVert_{L^\infty}^2, \\
\lvert\nabla\ell(x) \nabla^2 \Phi(y) \nabla \ell(x)\rvert
&= \lvert\nabla\ell_0(x)^\top \nabla^2\Phi_0(y) \nabla\ell_0(x)\rvert
\leqslant \lVert\nabla^2\Phi_0\rVert_{L^\infty}^2
\lVert\nabla \ell_0\rVert_{L^\infty}.
\end{align*}
All the conditions of Theorem~\ref{thm:dist-si-stationary-mkv-invariant}
are satisfied.
Since for all $y \in \mathbb R^D$,
\begin{align*}
\kappa_{\Phi_0}\lvert\langle\ell_0,m-m_*\rangle\rvert^2
&\leqslant
\langle\ell_0,m-m_*\rangle^\perp\nabla^2\Phi_0(y)
\langle\ell_0,m-m_*\rangle \\
&= \langle\ell,m-m_*\rangle^\perp\nabla^2\Phi(y)
\langle\ell,m-m_*\rangle,
\end{align*}
the claim of the proposition follows from the second case stated in the theorem.
\end{proof}

\begin{rem}[Open question]
Applying further the convergence result of Theorems~\ref{thm:si-contraction},
we can obtain a bound on the difference between the marginal distribution
of the non-stationary self-interacting diffusion \eqref{eq:si} and the
invariant measure of the McKean--Vlasov dynamics \eqref{eq:mkv}.
Specifically, let $(X_t, m_t)$ denote
the self-interacting process \eqref{eq:si}.
Theorem~\ref{thm:si-contraction} yields
\[
W_1\bigl(\Law(X_t, m_t),\Law(X, m)\bigr)\leqslant Ce^{-c t},
\]
where $C$, $c$ are the contractivity constants in the theorem.
Let $\varphi : \mathbb R^K \to \mathbb R$ be a $1$-Lipschitz test function.
Combining this with Proposition~\ref{prop:example-dynamics},
we can bound the following $L^1$ simulation error:
\[
\Expect [ \lvert \langle \varphi \circ \ell, m_t - m_*\rangle\rvert]
= \Expect [ \lvert \langle \varphi \circ \ell, m_t - m + m - m_*\rangle\rvert]
= O (e^{-c t} + \sqrt \lambda).
\]
As noted in Remark~\ref{rem:si-contraction-rate},
the contraction rate $c$ depends on $1/\lambda$ exponentially,
rendering the above error bound impractical for applications.

This naturally raises the question of whether the method and results
of Theorem~\ref{thm:dist-si-stationary-mkv-invariant},
which address the static case
(i.e., the comparison between $\Law(X,m)$ and $m_*\otimes\delta_{m_*}$),
can be extended to the dynamical setting
(comparing $\Law(X_t, m_t)$ and $m_*\otimes\delta_{m_*}$).
Unfortunately, we are currently unable to establish the entropy estimate in
Section~\ref{sec:entropy-estimate} for the parabolic problem, so our approach
does not yet apply in this context.
We leave this as an open problem for future research.
\end{rem}

\section{Numerical application to training two-layer neural networks}
\label{sec:numerical}

Let us recall the structure of two-layer neural networks and introduce
the mean field model for it.
Consider an \emph{activation function} \(\varphi : \mathbb R \to \mathbb R\)
satisfying
\begin{equation*}
\begin{gathered}
\text{$\varphi$ is continuous and non-decreasing,} \\
\lim_{\theta\to -\infty} \varphi(\theta) = 0,
\quad \lim_{\theta\to +\infty} \varphi(\theta) = 1.
\end{gathered}
\end{equation*}
Define \(S = \mathbb R \times \mathbb R^{\din} \times \mathbb R\), where the
\emph{neurons} take values.
For each neuron \(x = (c,a,b) \in S\)
we define the \emph{feature map}:
\begin{equation*}
\mathbb R^{d_{\din}} \ni z \mapsto f(x; z) \coloneqq \tau(c) \varphi (a \cdot z
+ b) \in \mathbb R,
\end{equation*}
where \(\tau : \mathbb R \to [-L, L]\) is a \emph{truncation function} with the
\emph{truncation threshold} \(L \in (0,+\infty)\).
The two-layer neural network is nothing but
the averaged feature map parameterized by
\(N\) neurons
\(x^1, \ldots, x^N \in S\):
\begin{equation}
\label{eq:nn-ps-feature}
\mathbb R^{\din} \ni z \mapsto f^N (x^1,\ldots,x^N ; z) = \frac 1N
\sum_{i=1}^N f(x^i; z) \in \mathbb R.
\end{equation}
The training of neural network aims to minimize the distance between the
averaged output \eqref{eq:nn-ps-feature}
for $K$ data points $(z_1, \ldots, z_K)$ and their labels $(y_1, \ldots, y_K)$,
that is, the objective function of the minimization reads
\begin{equation}
\label{eq:nn-ps-loss}
F_\textnormal{NNet}^N (x^1,\ldots,x^N) =
\frac{N}{2K} \sum_{i=1}^K
\bigl\lvert y_i - f^N(x^1,\ldots,x^N; z_i)\bigr\rvert^2 .
\end{equation}
This objective function is of high dimension and non-convex,
and this difficulty motivates the recent studies,
see for example \cite{MMNMF, ChizatBachGlobal, HRSS} among others,
to lift the finite-dimensional optimization \eqref{eq:nn-ps-loss}
to the space of probability measures
and to consider the following mean field optimization:
\begin{equation*}
\inf_{m\in \mathcal{P}_2(S)} F_\textnormal{NNet}(m), \quad
\text{where } F_\textnormal{NNet}(m) \coloneqq
\frac{1}{2K} \sum_{i=1}^K \bigl\lvert y_i - \Expect^{X \sim m} [f(X; z_i)]
\bigr\rvert^2.
\end{equation*}
The mean field loss functional \(F_\textnormal{NNet}\) is apparently convex.
Given that the activation and truncation functions $\varphi$, $\tau$
have bounded derivatives of up to fourth order,
it has been proved in \cite[Proposition 4.4]{ulpoc} that the minimum
of the entropy-regularized mean field optimization problem
\[
\inf_{m\in \mathcal{P}_2(S)} F_\textnormal{NNet}(m)
+ \frac{\sigma^2}{2}H\bigl(m \big|\mathcal{N}(0,\gamma^{-1})\bigr)
\]
can be attained by the invariant measure of the mean field Langevin dynamics:
\begin{equation}\label{eq:MFL}
\dd X_t = -D_m F\bigl(\Law(X_t), X_t\bigr) \dd t + \sigma \dd W_t,
\end{equation}
where the mean field potential functional reads
\[
F (m) \coloneqq F_\textnormal{NNet}(m)
+ \frac{\sigma^2\gamma}{2} \int |x|^2 m(\dd x).
\]
By defining
\begin{align*}
&\ell^0(x)\coloneqq |x|^2,\qquad
\ell^i(x)\coloneqq f(x, z_i) - y_i,\quad\text{for}~i =1,\ldots, K,\\
&R^{K+1}\ni \theta = (\theta_0,\theta_1,\ldots, \theta_K)
\mapsto \Phi(\theta)\coloneqq \frac{\sigma^2\gamma}{2} \theta_0
+ \frac{1}{2K}\sum_{i=1}^K|\theta_i|^2,
\end{align*}
we note that the mean field potential functional is of the form:
\[
F(m) = \Phi(\langle \ell, m\rangle),
\quad\text{where}~\ell\coloneqq (\ell^0, \ell^1,\ldots, \ell^K),
\]
as in the gradient case investigated in Sections~\ref{subsec:gradientcase}.

In order to simulate the invariant measure
of the mean field Langevin dynamics \eqref{eq:MFL},
one usually turns to the corresponding $N$-particle system:
\begin{equation}\label{eq:N particle}
\dd X^j_t = \bigl(-\nabla_j F_\textnormal{NNet}^N (X^1_t, \ldots, X^N_t)
- \sigma^2\gamma X^j_t \bigr) \dd t +\sigma \dd W^j_t,
\quad\text{for}~j=1,\ldots,N
\end{equation}
The uniform-in-time propagation of chaos results obtained
in \cite{ulpoc, SNWUniform} ensure
that the marginal distributions $\Law(X^1_t, \ldots, X^N_t)$
of the $N$-particle system are close to
those of the mean field Langevin dynamics uniformly on the whole time horizon,
and can efficiently approximate mean field invariant measure
provided that $t$ and $N$ are both large enough.

Note that the $N$-particle system \eqref{eq:N particle} is
nothing but a regularized and noised gradient descent flow for $N$ neurons.
In contrast, the self-interacting diffusion
\begin{equation}
\label{eq:selfinteracting_numerics}
\begin{aligned}
\dd X_t &= -\sum_{i=0}^K \nabla_i \Phi\bigl(Y^0_t, Y^1_t,\ldots, Y^K_t\bigr)
\nabla \ell^i(X_t) \dd t
+ \sigma \dd W_t \\
&= -\biggl(\frac{1}{K}\sum_{k=1}^K Y_t^k\nabla f(X_t,z_k)
+\sigma^2\gamma X_t\biggr) \dd t + \sigma \dd W_t, \\
\dd Y^i_t &= \lambda \bigl( \ell^i(X_t)
- Y_t^i\bigr)\dd t,\quad\text{for}~i=1,\ldots,K,
\end{aligned}
\end{equation}
introduces an innovative alternative algorithm
for training two-layer neural networks,
in which the algorithm iterations impact only a single neuron.

\paragraph{Setup.}
We aim to train a neural network to approximate
the non-linear elementary function
$\mathbb{R}^2\ni z= (z_1, z_2)
\mapsto g(z)\coloneqq\sin{2\pi z_1} + \cos{2\pi z_2}$.
Note that in this numerical example $\din=2$
and therefore $S=\mathbb R^{1+2+1}$.
We draw $K$ points $\{z_i\}^K_{k=1}$ sampled according
to the uniform distribution on $[0,1]^2$ and compute the corresponding labels
$y_k=g(z_k)$ to form our training data $\{z_k, y_k\}_{k=1}^K$.
We fix the truncation function $\tau$
by $\tau(\theta)=(\theta \wedge 30) \vee -30$
and the sigmoid activation function $\varphi$
by $\varphi(\theta)=1/\bigl(1+\exp(-\theta)\bigr)$.
The Brownian noise has volatility $\sigma$ such that $\sigma^2\!/2=0.05$,
and the regularization constant $\gamma$ is fixed
at $\gamma=0.0025$ in the experiment.
The initial value $X_0 = (C_0, A_0, B_0)$,
taking values in $ S= \mathbb R^{1+2+1}$,
is sampled from the normal distribution
 $m_0=\mathcal{N}(0,10^2 \times \Id_{4 \times 4})$
in four dimensions.
To observe the impact of the self-interacting coefficient $\lambda$,
we run the simulation of \eqref{eq:selfinteracting_numerics}
for different $\lambda$ equal to $4^{-p}$ for $p=4, \ldots, 8$.
Furthermore, to compare with these results with fixed $\lambda$,
we choose the non-increasing piecewise constant function
$\lambda_\textnormal{a}$
such that $\lambda_\textnormal{a}(t)=4^{-i}$ on successive intervals
of length $\delta T_i=4^i$, for $i=2,\ldots,11$,
and train the neural network along the annealing scheme:
\begin{equation}
\label{eq:annealing_numerics}
\begin{aligned}
\dd X_t &= -\biggl(\frac{1}{K}\sum_{k=1}^K Y_t^k\nabla f(X_t,z_k)
+\sigma^2\gamma X_t\biggr) \dd t +\sigma \dd W_t,\\
\dd Y^i_t&= \lambda_\textnormal{a}(t) \bigl( \ell^i(X_t) - Y_t^i\bigr) \dd t,
\quad\text{for}~i=1,\ldots,K.
\end{aligned}
\end{equation}
We simulate both the constant and dynamic-$\lambda$
self-interacting diffusions \eqref{eq:selfinteracting_numerics},
\eqref{eq:annealing_numerics} by the Euler scheme,
as described in Section~\ref{sec:algorithm},
on a long time horizon till the terminal time $T=10^6$,
with the discrete time step $\delta t=0.5$.

\begin{figure}[!tb]
\centering
\includegraphics[width=0.75\textwidth]{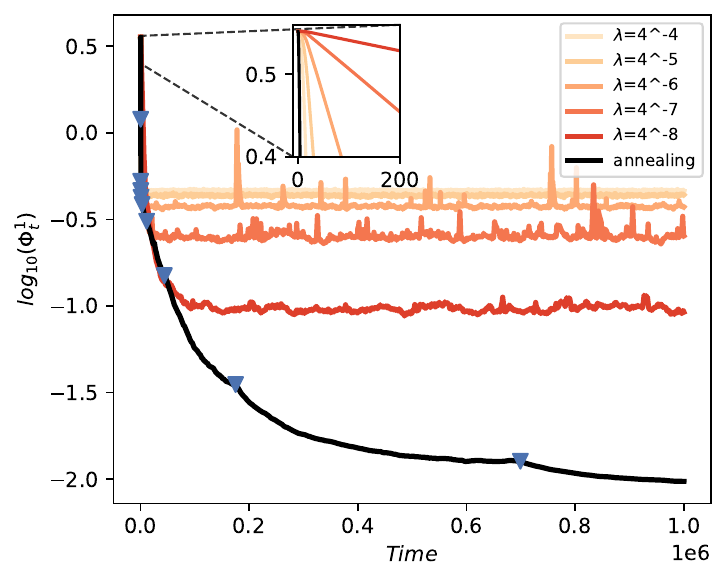}
\caption{Averaged over 100 repetitions losses for fixed values
of $\lambda$ and for discrete annealing.}
\label{fig:results}
\end{figure}

\paragraph{Results and discussions.}
We repeat the simulations mentioned above
for fixed $\lambda$'s
and dynamic $\lambda_\textnormal{a}$ all
for $100$ times and compute the averaged loss $\Phi(Y_t)$ at each time $t$
and plot its evolution in Figure~\ref{fig:results}.
On the annealing scheme curve, we include triangles to visually
indicate the points at which there are changes on the values
of \(\lambda_\textnormal{a}\).
We observe that the value of $\Phi(Y_t)$ first decays exponentially,
and the speed of initial decay depending on the value of $\lambda$.
More precisely, the bigger is the value of $\lambda$,
the faster is the initial decay.
In particular, this remains true for the annealed process
as it starts from a bigger value $\lambda_\textnormal{a}(0)$.
We notice that such phenomenon is in line
with our theoretical results in Theorem~\ref{thm:si-contraction}.
In the long run, when fixing $\lambda$, the value of $\Phi(Y_t)$ stabilizes
at a level sensible with respect to the value of $\lambda$.
We notice that the smaller is the value of $\lambda$,
the better is the final performance.
These facts are in line with our discovery
in Theorems~\ref{thm:si-contraction}
and \ref{thm:dist-si-stationary-mkv-invariant}.
Finally, the training exhibits the best performance
when implementing the discrete annealing.
The loss $\Phi(Y_t)$ continues to decrease
as the value of $\lambda_\textnormal{a}(t)$ diminishes.

\section{A Curie--Weiss model}
\label{sec:curie-weiss}

In this section, we present and study a simple Curie--Weiss model,
i.e., a mean field model of ferromagnets,
which has possibly multiple invariant measures.
In particular, we show that in this case,
the last claim of Theorem~\ref{thm:dist-si-stationary-mkv-invariant} provides
informations on the concentration of the self-interacting stationary measure.

Let $\ell_0 : \mathbb R \to \mathbb R$
be a smooth, odd, increasing function in $\mathcal C^1 \cap W^{1,\infty}$.
For a probability measure $m \in \mathcal P(\mathbb R)$,
consider the mean field energy
\[
F (m) = - \frac 12 \biggl(\int \ell_0 (x)m(\dd x)\biggr)^{\!2}
+ \frac 12 \int x^2 m(\dd x).
\]
By setting
\begin{align*}
\ell(x) &\coloneqq \bigl( \ell_0 (x), \lvert x\rvert^2\!/2 \bigr)^\top, \\
\Phi(y_0, y_1) &\coloneqq - \frac 12 \lvert y_0\rvert^2 + y_1,
\end{align*}
we have $F(m) = \Phi(\langle \ell, m\rangle)$.
So we are in the cylindrical setting of Assumption~\ref{assu:finite-dim-gd}
with the range set of $\ell$ being defined by
\[
\mathcal K \coloneqq [ - \lVert \ell_0\rVert_{L^\infty},
\lVert \ell_0\rVert_{L^\infty} ] \times \mathbb R.
\]
Moreover, as the corresponding intrinsic derivative, or the drift, writes
\[
b(m,x) = - D_m F(m,x)
= \langle \ell_0, m\rangle \ell'_0(x) - x,
\]
we can verify Assumption~\ref{assu:si-contraction}
with the modulus of continuity
\[
\omega (r) = \sup_{x,x' \in \mathbb R : \lvert x - x'\rvert \leqslant r}
\lvert \ell_0 (x) - \ell_0(x')\rvert.
\]
This implies, by Corollary~\ref{cor:si-stationary-exist-unique},
that the self-interacting process has a unique invariant measure,
which we denote by $\rho^\lambda = \rho$.
Arguing as in the proof of Proposition~\ref{prop:example-dynamics},
we can verify the conditions
of Theorem~\ref{thm:dist-si-stationary-mkv-invariant}
for the choice of $\Phi$, $\ell$ and $\rho$ above.
Especially, the probability measures
\[
\hat m_{(y_0,y_1)}(\dd x)
\propto \exp \biggl( -\frac 12 \lvert x\rvert^2
+y_0 \ell_0(x) \biggr) \dd x
\]
satisfy a uniform LSI thanks to the boundedness of $y_0$
and the Holley--Stroock perturbation lemma.

Before applying Theorem~\ref{thm:dist-si-stationary-mkv-invariant},
we first give a characterization of the invariant measure
for the McKean--Vlasov dynamics \eqref{eq:mkv}.
We discuss a special case and then move to general discussions.

\begin{enumerate}
\item If $\lVert \ell' \rVert_{L^\infty} < 1$,
then by the result of \cite{WangLandauType},
we already know that the McKean--Vlasov dynamics \eqref{eq:mkv}
has a unique invariant measure,
which is the centered Gaussian of variance $1/2$, i.e., $\mathcal N(0,1/2)$.
This case corresponds to the weak interaction regime studied in
\cite{DJLEmpApprox}.

\item
In the general case where $\lVert \ell' \rVert_{L^\infty}$ is not restricted,
the invariant measures of \eqref{eq:mkv} correspond to
fixed points of the one-dimensional mapping
\[\mathbb R \ni y_0 \mapsto \Pi_0(y_0) = \hat y_0 \in \mathbb R\]
given by
\[
\Pi_0(y_0) = \frac{\int \ell_0(x) \exp \bigl( 2y_0 \ell_0(x)
- \lvert x\rvert^2 \bigr) \dd x }
{\int \exp \bigl( 2y_0 \ell_0(x) - \lvert x\rvert^2 \bigr) \dd x }
\]
That is to say, if $y_0$ satisfies $\Pi_0(y_0) = y_0$,
then the probability measure proportional to
$\exp \bigl( 2y_0 \ell_0(x) - \lvert x\rvert^2 \bigr) \dd x$
is invariant to \eqref{eq:mkv}; and vice versa.
Due to the oddness of $\ell_0$, the mapping $\Pi_0$ is odd.
In particular, $\Pi_0(0) = 0$
and we know that $\mathcal N(0,1/2)$ is always invariant.

If $\Pi_0'(0) > 1$,
then by the fact that $\lVert \Pi_0\rVert_{L^\infty}
\leqslant \lVert \ell_0\rVert_{L^\infty} < \infty$
and the intermediate value theorem,
there exists $y_0 > 0$ such that
$\Pi_0 (y_0) = y_0$ and $\Pi_0(-y_0) = -y_0$.
That is to say there exists at least three invariant measures,
two of which are, in physicists' language, ``symmetry-breaking'' phases,
and the centered Gaussian measure is the ``symmetry-preserving'' phase.
Moreover, the centered measure corresponding to $y_0=0$ should be unstable
as it is a local maximum point for the free energy landscape
$y_0 \mapsto \frac 12 y_0^2 - \frac 12
\log \int \exp \bigl( 2y_0 \ell_0(x) - \lvert x\rvert^2 \bigr) \dd x$.
\end{enumerate}

We now turn to the study of the stationary measure $\rho^\lambda$
of the self-interacting process \eqref{eq:si}.
Since $\Phi(y) = - \lvert y_0\rvert^2\!/2 + y_1$
is linear in its second coordinate,
the last claim of Theorem~\ref{thm:dist-si-stationary-mkv-invariant}
implies that
\[
\int \lvert \hat y_0 - y_0 \rvert^2 \rho^2(\dd y)=
-\int (\hat y-y)^\top\nabla^2\Phi(y)(\hat y-y) \rho^2(\dd y)\leqslant C \lambda,
\]
where $C$ is a constant depending only on $\ell_0$
and as we recall, $y_0$ is the first coordinate of $y$.
In other words, the stationary measure $\rho^\lambda$
solves the self-consistency equation
\[
\hat y_0 = y_0
\]
up to an error of order $O(\sqrt\lambda)$.
Denote the set of fixed points of $y_0 \mapsto \Pi_0(y_0) = \hat y_0$ by $S$.
If the set $S$ is finite, and if for each $s \in S$ we have
\[
\Pi_0'(s) \neq 1,
\]
then there exists $c > 0$ such that for all $y_0 \in \mathbb R$,
\[
\lvert \hat y_0 - y_0 \rvert \geqslant c \min\bigl(d(y_0, S),1\bigr),
\]
where $d(\cdot,S)$ indicates the distance to the set $S$.
In this case, we have
\[
\int \min\bigl(d(y_0,S),1\bigr)^2 \rho^2(\dd y) = O(\lambda).
\]
That is to say, for small $\lambda$,
the random variable $Y = (Y_0,Y_1)$, distributed as the second component
of the stationary measure $\rho^\lambda$,
has $Y_0$ concentrated near the solutions to the self-consistency equation,
which correspond to invariant measures of the McKean--Vlasov dynamics.
However, the last claim of Theorem~\ref{thm:dist-si-stationary-mkv-invariant}
is not sufficient to show that $Y_0$ is only concentrated around,
or in a way ``selects'', the physically stable phases
that are minimizers of a free energy functional.
We refer readers to \cite{BRSIDiffusions3} for qualitative results
on such selection mechanism.

\section{Proof of Theorem~\ref{thm:si-contraction}}
\label{sec:si-contraction}

We first note that as the metric space
$\bigl(\mathcal P(\mathbb R^d), W_\omega\bigr)$
is separable, we do not have measurability issues.
We refer readers to \cite[Chapter~2]{LedouxTalagrandProbaInBanach}
for details.

Recall that the self-interacting dynamics \eqref{eq:si} writes
\[
\begin{aligned}
\dd X_t &= b(m_t, X_t) \dd t + \dd B_t, \\
\dd m_t &= \lambda (- m_t + \delta_{X_t}) \dd t
\end{aligned}
\]
and similarly for the other dynamics $(X', m')$
driven by another Brownian motion $B'$.
Fix $n \in \mathbb N$.
Let $\refc : \mathbb R^d \times \mathbb R^d \to [0,1]$
be a Lipschitz continuous function such that
$\sync \coloneqq \sqrt{1 - \refc^2}$ is also Lipschitz continuous and
\[
\refc(x,x') = \begin{cases}
1 & \text{if}~|x - x'| \geqslant 2n^{-1}, \\
0 & \text{if}~|x - x'| \leqslant n^{-1}.
\end{cases}
\]
We couple the two dynamics $(X,m)$, $(X',m')$ by
\begin{align*}
\dd B_t &= \refc (X_t, X'_t) \dd B^{\refc}_t
+ \sync(X_t, X'_t) \dd B^{\sync}_t, \\
\dd B'_t &= \refc (X_t, X'_t) \bigl(1 - 2e_te_t^\top\bigr) \dd B^{\refc}_t
+ \sync(X_t, X'_t) \dd B^{\sync}_t,
\end{align*}
where $B^{\refc}$, $B^{\sync}$ are independent Brownian motions
and $e_t$ is the $d$-dimensional vector defined by
\[
e_t = \begin{cases}
\frac{X_t - X'_t}{|X_t - X'_t|} & \text{if}~X_t \neq X'_t, \\
(1,0,\ldots,0)^\top & \text{otherwise}.
\end{cases}
\]

Subtracting the dynamical equations of $X$, $X'$
and denoting $\delta X = X - X'$, we obtain
\[
\dd \delta X_t = \bigl(b(m_t, X_t) - b(m'_t, X'_t)\bigr) \dd t
+ 2 \refc(X_t, X'_t) e_t \dd W_t,
\]
where $W_t \coloneqq \int_0^t e_t^\top \dd B^{\refc}_t$
and is a one-dimensional Brownian motion by Lévy's characterization.
The absolute value of the semimartingale $\delta X_t$
admits the decomposition
$\dd |\delta X_t| = \dd A^{|\delta X|}_t + \dd M^{|\delta X|}_t$ with
\begin{align*}
\dd A^{|\delta X|}_t &\leqslant - |\delta X_t| \kappa(|\delta X_t|) \dd t
+ L W_\omega(m_t,m'_t) \dd t, \\
\dd M^{|\delta X|}_t &= 2 \refc(X_t, X'_t) \dd W_t.
\end{align*}
For the $m$ variable, we have
\[
\dd (m_t - m'_t) = - \lambda (m_t - m'_t) \dd t
+ \lambda (\delta_{X_t} - \delta_{X'_t}) \dd t.
\]
Thus, by considering the (random) $W_\omega$-optimal coupling
at each $t$, we get
\[
\dd W_\omega (m_t, m'_t) \leqslant - \lambda W_\omega(m_t, m'_t) \dd t
+ \lambda \omega (\lvert \delta X_t\rvert) \dd t.
\]
Hence the process
\[
r_t = \lvert\delta X_t\rvert + \frac{2 L}{\lambda} W_\omega(m_t, m'_t)
\]
admits the decomposition $\dd r_t = \dd A^r_t + \dd M^r_t$ with
\begin{align*}
\dd A^r_t
&\leqslant \bigl( - \lvert\delta X_t\rvert \kappa(|\delta X_t|)
+ 2 L \omega(|\delta X_t|) - L W_\omega(m_t,m'_t) \bigr) \dd t, \\
\dd M^r_t &= 2 \refc (X_t, X'_t) \dd W_t.
\end{align*}

Let $f : [0,\infty)\to[0,\infty)$ be a $\mathcal C^2$ function
to be determined in the following.
We define $\rho_t = f(r_t)$.
The process $\rho_t$ admits the decomposition
$\dd\rho_t = \dd A^\rho_t + \dd M^\rho_t$ with
\begin{align*}
\dd A^\rho_t &\leqslant \bigl( - |\delta X_t| \kappa(|\delta X_t|)
+ 2L\omega(|\delta X_t|) - L W_\omega(m_t,m'_t) \bigr)
f'_-(r_t) \dd t \\
&\hphantom{\leqslant{}} \quad + 2\refc (X_t,X'_t)^2 f''(r_t) \dd t \\
&\leqslant - r_t \tilde K(r_t) f'_-(r_t) \dd t
+ 2 \refc (X_t, X'_t)^2 f''(r_t)\dd t
\end{align*}
for the function $\tilde K : (0,\infty) \to \mathbb R$ defined by
\begin{align*}
\tilde K(r)
&\coloneqq \inf_{\substack{x + 2L\lambda^{-1} y = r \\ x, y > 0}}
\frac{x \kappa(x) - 2L \omega(x) + Ly}{r} \\
&\geqslant \inf_{\substack{x + 2L\lambda^{-1} y = r \\ x, y > 0}}
\frac{\kappa_0 x + Ly - M_b - 2L M_\omega}{r} \\
&\geqslant \min \biggl( \kappa_0, \frac{\lambda}{2} \biggr)
- \frac{M_b + 2L M_\omega}{r} \\
&\eqqcolon K_0 - \frac{M}{r}
\eqqcolon K(r).
\end{align*}
Thus, we have shown
\[
\dd A^\rho_t \leqslant - r_t K(r_t) f'_-(r_t) \dd t
+ 2 \refc (X_t, X'_t)^2 f''(r_t) \dd t.
\]
Following Du et al.\ \cite{DJLSequentialPOC},
we choose the function $f : [0,\infty) \to [0,\infty)$ by requiring
\[
f'(r) = \frac 12 \int_r^\infty s \exp \biggl(
- \frac 12 \int_r^s \tau K(\tau) \dd\tau \biggr) \dd s.
\]
and $f(0) = 0$.
The function $f$ is indeed $\mathcal C^2$ continuous,
and, according to \cite[Lemma 5.1]{DJLSequentialPOC},
it is also concave and satisfies
\[
2f''(r) - rK(r) f'(r) + r = 0
\]
and
\[
\frac 1{K_0} \leqslant f'(r),\, \frac{f(r)}{r} \leqslant f'(0)
\]
for all $r > 0$.
Plugging in the expression for $K$, we obtain
\begin{align*}
f'(0)
&= \frac 12 \int_0^\infty s
\exp \biggl( - \frac{K_0s^2}{4} + \frac{Ms}{2} \biggr) \dd s \\
&= \frac 12 \exp \biggl( \frac{M^2}{4K_0}\biggr) \int_0^\infty s
\exp \biggl( - \frac{K_0(s - \frac M{K_0})^2}{4} \biggr) \dd s \\
&= \frac 12 \exp \biggl( \frac{M^2}{4K_0}\biggr) \int_{- M/\sqrt{2K_0}}^\infty
\Biggl( \frac{2t}{K_0} + M \frac{2^{1/2}}{K_0^{3/2}} \Biggr)
\exp \biggl( - \frac{t^2}{2} \biggr) \dd t \\
&\leqslant \frac 12 \exp \biggl( \frac{M^2}{4K_0}\biggr)
\Biggl( \frac{2}{K_0} \exp \biggl( - \frac{M^2}{4K_0} \biggr)
+ 2M \frac{\pi^{1/2}}{K_0^{3/2}} \Biggr) \\
&\leqslant \frac{1}{K_0}
+ \frac{2M}{K_0^{3/2}}\exp \biggl( \frac{M^2}{4K_0} \biggr).
\end{align*}

For $|\delta X_t| \geqslant 2n^{-1}$, we have $\refc_t (X_t, X'_t) = 1$,
and by the previous construction,
\[
\dd f(r_t) = \dd A^{\rho}_t + \dd M^{\rho}_t
\leqslant - r_t \dd t + \dd M^{\rho}_t
\leqslant - \frac{f(r_t)}{K_0^{-1} + 2M K_0^{-3/2} \exp (M^2\!/4K_0)} \dd t
+ \dd M^{\rho}_t.
\]
For $|\delta X_t| < 2n^{-1}$, we proceed differently.
Let $\omega_b$ denote the uniform modulus of equicontinuity
of the vector fields $x \mapsto b(m,x)$.
The absolute continuous part of $r_t$ satisfies
\begin{align*}
\dd A^r_t &\leqslant \bigl( \omega_b(2n^{-1}) + 2L \omega(2n^{-1})
- L W_\omega(m_t,m'_t) \bigr) \dd t \\
&= - L W_\omega(m_t,m'_t) \dd t + o(1) \dd t \\
&\leqslant - \frac{\lambda}{2} r_t \dd t
+ o(1) \dd t,
\end{align*}
where $o(1)$ denotes a term that tends to $0$ when $n \to \infty$.

Now we combine the two cases.
Define a sequence of functions by
\[
f_n(r) = \begin{cases}
f(r), &\text{if $r \geqslant 2n^{-1}$,} \\
\frac{f(2n^{-1})}{2n^{-1}}r, &\text{if $r < 2n^{-1}$.}
\end{cases}
\]
Each function $f_n$ is concave and satisfies, by the arguments above,
\[
\dd \Expect[f_n(r_t)]
\leqslant - c \Expect[f_n(r_t)] \dd t + o(1) \dd t,
\]
where
\begin{align*}
c &\coloneqq \min \biggl(
\frac{1}{K_0^{-1} + 2MK_0^{-3/2} \exp(M^2\!/4K_0)}, \frac{\lambda}{2}
\biggr) \\
&= \frac{1}{K_0^{-1} + 2MK_0^{-3/2} \exp(M^2\!/4K_0)}.
\end{align*}
Applying Grönwall's lemma, we obtain
\[
\Expect [f_n(r_t)] \leqslant e^{-ct} \Expect [f_n(r_0)] + o(1).
\]
Since $\lim_{n\to\infty}\Expect[f_n(r_t)] = \Expect[f(r_t)]$
by dominated convergence,
taking the limit $n \to \infty$ completes the proof.
\qed

\section{Proof of Theorem~\ref{thm:dist-si-stationary-mkv-invariant}}
\label{sec:gd}

This section consists of four subsections.
We show a series of intermediate,
and sometimes technical, lemmas and propositions in the first three subsections
before proving the theorem in the last subsection.

\subsection{Elliptic equation for the stationary measure}

We first note that the stationary measure $P$
solves a partial differential equation in the following weak sense.

\begin{prop}
\label{prop:si-stationary-pde-weak-sol}
Let $\mathcal C^{1,2}_\textnormal{b}$ denote the class of functionals
$\phi : \mathcal P (\mathbb R^d) \times \mathbb R^d \to \mathbb R$
with continuous spatial derivatives $\nabla_x \phi$, $\nabla_x^2 \phi$
and a bounded linear functional derivative
$\frac{\delta \phi}{\delta m}:
\mathcal P(\mathbb R^d) \times \mathbb R^d \times \mathbb R^d
\to \mathbb R$ satisfying
\begin{align*}
&\forall m \in \mathcal P(\mathbb R^d),~\forall x \in \mathbb R^d,&
\lvert \nabla_x \phi(m,x) \rvert
+ \bigl| \nabla_x^2 \phi(m,x) \bigr| &\leqslant C, \\
&\forall m \in \mathcal P(\mathbb R^d),~\forall x,x' \in \mathbb R^d,&
\biggl| \frac{\delta \phi}{\delta m} (m,x,x') \biggr| &\leqslant C,
\end{align*}
for some $C \geqslant 0$.
Under the settings of the theorem,
let $(X,m)$ be a random variable distributed as the stationary measure $P$.
Then we have, for all $\phi \in \mathcal C^{1,2}_\textnormal{b}$,
\begin{multline*}
\Expect\biggl[
\frac 12 \Delta_x \phi (m, X) + b(m, X) \cdot \nabla_x \phi(m,X) \\
+ \lambda \int \frac{\delta\phi}{\delta m}(m,X,x') (\delta_X - m)(\dd x')
\biggr] = 0.
\end{multline*}
\end{prop}

We omit the proof of the proposition, which follows directly
from expanding the difference
$\Expect [\phi(m_t, X_t)] - \Expect [\phi(m_0, X_0)]$
by It\=o-type calculus.

The infinite-dimensional nature of the PDE above makes its analysis difficult,
and in the following we approach the problem by studying
a finite-dimensional projection of it.
Under Assumption~\ref{assu:finite-dim-gd}, define the functions
\begin{align*}
\beta(y,x) &= - \nabla \Phi(y) \cdot \nabla \ell(x)
= - \sum_{\nu = 1}^D \nabla_\nu \Phi(y) \nabla \ell^\nu(x), \\
V(y,x) &= \nabla \Phi(y) \cdot \ell (x)
= \sum_{\nu = 1}^D \nabla_\nu \Phi(y) \ell^\nu (x).
\end{align*}
They verify $\beta(y,x) = - \nabla_x V(y,x)$.
Note that, if $m_y$ is a measure satisfying $\int \ell(x) m_y(\dd x) = y$,
then we have
\begin{align*}
\beta (y,x) &= b(m_y, x), \\
V(y,x) &= \frac{\delta F}{\delta m}(m_y,x).
\end{align*}
Let $(X,m)$ be distributed as the stationary measure $P$
and consider the random variable
$Y = \langle \ell, m\rangle = \int \ell(x) m(\dd x)$
valued in $\mathcal K$.
Applying Proposition~\ref{prop:si-stationary-pde-weak-sol}
to functionals of the following form:
\[
\phi(m,x) = \varphi \biggl( x, \int \ell(x') m(\dd x')\biggr),
\]
where $\varphi \in \mathcal D (\mathbb R^d \times \mathbb R^D)$,
we get that the joint distribution $\rho = \rho^\lambda \coloneqq \Law(X, Y)
\in \mathcal P(\mathbb R^d \times \mathcal K)$
satisfies the stationary degenerate Fokker--Planck equation
\begin{equation}
\label{eq:fp-static}
\frac 12 \Delta_x \rho
- \nabla_x \cdot \bigl(\beta(y,x) \rho\bigr)
- \lambda \nabla_y \cdot \Bigl( \bigl( \ell(x) - y \bigr) \rho \Bigr)
= 0.
\end{equation}
in the sense of distributions.
The fact that $P$ has finite fourth moment implies that
its projection $\rho$ satisfies the following moment condition:
\begin{equation}
\label{eq:rho-moment-condition}
\int ( \lvert x\rvert^4 + \lvert y \rvert^2 ) \rho(\dd x\dd y) < \infty.
\end{equation}
From the Fokker--Planck equation \eqref{eq:fp-static},
we get the following result.

\begin{lem}
\label{lem:grad-invariant}
Under the setting of the theorem, for every function
$\varphi \in \mathcal C^1(\mathbb R^{D};\mathbb R)$
whose gradient $\nabla \varphi$ is of linear growth,
we have
\[
\iint \nabla \varphi(y) \cdot \bigl( \ell(x) - y \bigr) \rho(\dd x\dd y)
= 0.
\]
\end{lem}

\begin{proof}[Proof of Lemma~\ref{lem:grad-invariant}]
Since its gradient $\nabla \varphi$ is of linear growth,
the function $\varphi$ is of quadratic growth.
Thanks to the fact that $\rho$ satisfies the moment condition
\eqref{eq:rho-moment-condition},
we can take the duality with $\varphi$ in the static Fokker--Planck equation
\eqref{eq:fp-static}, from which the desired result follows.
\end{proof}

\subsection{Entropy estimate}
\label{sec:entropy-estimate}

In this subsection, we show an entropy estimate for the stationary
measure $\rho$ by studying the Fokker--Planck equation \eqref{eq:fp-static}.

Denote the first and second marginal distributions
of $\rho$ by $\rho^{1}$, $\rho^{2}$ respectively.
For $\rho^2$-almost all $y \in \mathbb R^D$,
denote also the conditional distribution of the first variable
by $\rho^{1|2}(\cdot|y) : \mathcal B(\mathbb R^d) \to \mathbb R$.
Define
\begin{equation}
\label{eq:def-hat-rho}
\hat \rho (\dd x\dd y) \coloneqq \hat m_y (\dd x) \rho^2(\dd y)
= \frac 1{Z_y} \exp \bigl( - 2 V(y,x) \bigr) \dd x\,\rho^2(\dd y),
\end{equation}
for $Z_y = \int \exp \bigl( - 2V(y,x) \bigr) \dd x$.
Recall that $\hat m_y$ is the probability measure on $\mathbb R^d$
that satisfies a uniform LSI
according to Assumption~\ref{assu:m-hat-uniform-lsi}.

\begin{prop}
\label{prop:rho-hat-rho-entropy}
Under the setting of the theorem, we have
\[
H(\rho | \hat\rho)
\leqslant
\frac{C_\textnormal{LS}}{2}
\biggl(D + 2 \iint \bigl(\ell(x) - y\bigr)^\top \nabla^2\Phi(y)
\bigl(\ell(x) - y\bigr) \rho(\dd x\dd y)\biggr) \lambda,
\]
where $\hat \rho$ is the measure defined by \eqref{eq:def-hat-rho}.
\end{prop}

For convenience, we set
\begin{equation}
\label{eq:def-I}
I \coloneqq \iint \bigl(\ell(x) - y\bigr)^\top \nabla^2\Phi(y)
\bigl(\ell(x) - y\bigr) \rho(\dd x\dd y),
\end{equation}
so the claim of the proposition reads
\[
H(\rho | \hat \rho) \leqslant \frac{C_\textnormal{LS}}{2}(D+2I) \lambda.
\]

\begin{proof}[Proof of Proposition~\ref{prop:rho-hat-rho-entropy}]
Let $g^\varepsilon$ be the Gaussian kernel in $D$ dimensions:
\[
g^\varepsilon(y) = (2 \pi \varepsilon)^{-D/2}
\exp \biggl( - \frac{|y|^2}{2\varepsilon} \biggr).
\]
We define the partial convolution
$\rho^\varepsilon = \rho \star_y g^\varepsilon$,
and according to \eqref{eq:fp-static},
it solves the non-degenerate elliptic equation
\begin{equation}
\label{eq:fp-static-reg}
\frac 12 \Delta_x \rho^\varepsilon
- \nabla_x \cdot \bigl( \beta^\varepsilon(y,x) \rho^\varepsilon \bigr)
- \ell(x) \cdot \nabla_y \rho^\varepsilon
+ \lambda \varepsilon \Delta_y \rho^\varepsilon
+ \lambda \nabla_y \cdot (y \rho^\varepsilon) = 0,
\end{equation}
where $\beta^\varepsilon$ is defined by
\[
\beta^\varepsilon = \frac{(\beta \rho) \star_y
g^\varepsilon}{\rho^\varepsilon}.
\]
Indeed, we have
\begin{align*}
(y\rho) \star_y g^\varepsilon
&=\int y' \rho(x', y') g^\varepsilon (y - y') \dd y' \\
&= \int (y' - y + y) \rho(x', y') g^\varepsilon (y - y') \dd y' \\
&= \varepsilon \int \rho(x', y') \nabla_y g^\varepsilon(y-y') \dd y'
+ y \rho^\varepsilon \\
&= \varepsilon \nabla_y \rho^\varepsilon + y \rho^\varepsilon.
\end{align*}
Thus, $\bigl(\nabla_y \cdot (y\rho)\bigr) \star_y g^\varepsilon
= \nabla_y \cdot \bigl((y\rho) \star_y g^\varepsilon\bigr)
= \varepsilon \Delta_y \rho^\varepsilon
+ \nabla_y \cdot (y \rho^\varepsilon)$.
By \cite[Lemma 3.1.1]{BKRSFPKEq}, we have
$\lVert \beta^\varepsilon \rVert_{L^2(\rho^\varepsilon)}
\leqslant \lVert \beta \rVert_{L^2(\rho)} < \infty$.
Then we can apply \cite[Theorem 3.1.2]{BKRSFPKEq} to \eqref{eq:fp-static-reg}
and obtain that $\rho^\varepsilon \in W^{1,1}(\mathbb R^{d+D})$ and satisfies
\begin{multline*}
\frac 12 \iint \frac{|\nabla_x \rho^\varepsilon|^2}{\rho^\varepsilon}
+ \lambda \varepsilon \iint \frac{|\nabla_y
\rho^\varepsilon|^2}{\rho^\varepsilon} \\
= \iint \nabla_x \rho^\varepsilon \cdot \beta^\varepsilon
+ \lambda \iint \nabla_y \rho^\varepsilon \cdot \ell
- \lambda \iint \nabla_y \rho^\varepsilon \cdot y.
\end{multline*}
As the function $\ell$ depends only on the $x$ argument,
for the second term we have
\[
\iint \nabla_y \rho^\varepsilon \cdot \ell
= \int \biggl( \int \nabla_y \rho^\varepsilon(x,y) dy \biggr) \ell(x) \dd x
= \int 0 \cdot \ell(x) \dd x = 0.
\]
where the first equality is due to Fubini and the second to the fact that
$\nabla \rho^\varepsilon \in L^1_x(L^1_y)$.
For the last term, similarly, since the function
$\bigl( (x,y)\mapsto \rho^\varepsilon(x,y) y \bigr) \in W^{1,1}$,
we have
$\iint \nabla_y \cdot (\rho^\varepsilon y) = 0$ and therefore,
\[
- \iint \nabla_y \rho^\varepsilon \cdot y
= \iint (\nabla_y \cdot y) \rho^\varepsilon
= D.
\]
That is to say, we have established
\begin{equation}
\label{eq:rho-fisher-bound}
\frac 12 \iint \frac{|\nabla_x \rho^\varepsilon|^2}{\rho^\varepsilon}
+ \lambda \varepsilon \iint \frac{|\nabla_y
\rho^\varepsilon|^2}{\rho^\varepsilon}
= \iint \nabla_x \rho^\varepsilon \cdot \beta^\varepsilon + \lambda D.
\end{equation}
This equality implies a uniform-in-$\varepsilon$ bound on
$\iint |\nabla_x\rho^\varepsilon|^2/\rho^\varepsilon$ by Cauchy--Schwarz.
Using a sequence of functions in $\mathcal C^\infty_\textnormal{c}$
approaching $V(y,x)$ in \eqref{eq:fp-static},
we also get
\begin{multline}
\label{eq:rho-energy-bound}
\iint \frac 12 \beta \cdot \nabla_x \rho^\varepsilon
- \iint \lambda \varepsilon \nabla_y V \cdot \nabla_y \rho^\varepsilon
- \iint \beta \cdot \beta^{\varepsilon} \rho^\varepsilon \\
+ \iint \lambda \nabla_y V(y,x) \cdot \bigl( \ell(x) - y \bigr)
\rho^\varepsilon(\dd x\dd y)
= 0.
\end{multline}
The second term of \eqref{eq:rho-energy-bound} is upper bounded by
\begin{align*}
\lambda \varepsilon \iint \lvert\nabla_y V \cdot \nabla_y
\rho^\varepsilon\rvert
&\leqslant \lambda \varepsilon \lVert \nabla_y V \rVert_{L^2(\rho^\varepsilon)}
\biggl( \iint \frac{|\nabla_y \rho^\varepsilon|^2}{\rho^\varepsilon}
\biggr)^{\!1/2} \\
&\leqslant \sqrt{\lambda \varepsilon}
\lVert \nabla_y V \rVert_{L^2(\rho^\varepsilon)}
\sqrt{\frac 12 \iint \frac{|\nabla_x \rho^\varepsilon|^2}{\rho^\varepsilon} +
\frac 12 \lVert \beta^\varepsilon \rVert_{L^2(\rho^\varepsilon)}^2
+ \lambda D},
\end{align*}
where the second inequality is due to \eqref{eq:rho-fisher-bound}.
Using the uniform-in-$\varepsilon$ bound of
$\iint \lvert \nabla_x \rho^\varepsilon \rvert^2/ \rho^\varepsilon$,
we obtain that the second term of \eqref{eq:rho-energy-bound}
converges to $0$ when $\varepsilon \to 0$.
The third term of \eqref{eq:rho-energy-bound} satisfies
\[
\iint \beta \cdot \beta^\varepsilon \rho^\varepsilon
= \iint \beta \cdot \bigl((\beta \rho) \star g^\varepsilon\bigr)
= \iint (\beta \star g^\varepsilon) \cdot \beta \rho
\to \iint |\beta|^2 \rho
\]
when $\varepsilon \to 0$.
Hence, by \eqref{eq:rho-fisher-bound} and \eqref{eq:rho-energy-bound},
we have
\begin{align*}
\MoveEqLeft \frac 12 \iint |\nabla_x \log \rho^\varepsilon
- 2\beta|^2\rho^\varepsilon(\dd x\dd y) \\
&\leqslant \lambda D
+ \iint \nabla_x \rho^\varepsilon \cdot \beta^\varepsilon
+ \iint \beta \cdot \nabla_x \rho^\varepsilon
- 2 \iint \nabla_x \rho^\varepsilon \cdot \beta \\
&\hphantom{\leqslant{}}\quad
+ 2 \lambda \iint \nabla_y V(y,x) \cdot \bigl(\ell(x) - y\bigr)
\rho^\varepsilon(\dd x\dd y)
+ o(1)\\
&= \lambda D + 2 \lambda \iint \nabla_y V(y,x) \cdot \bigl(\ell(x) - y\bigr)
\rho^\varepsilon(\dd x\dd y) + o(1),
\end{align*}
where the last equality is due to the fact that
\[
\biggl\vert\iint \nabla_x \rho^\varepsilon \cdot
(\beta^\varepsilon-\beta)\biggr\vert
\leqslant
\biggl(\iint\frac{|\nabla_x\rho^\varepsilon|^2}{\rho^\varepsilon}\biggr)^{\!1/2}
\biggl(\iint|\beta - \beta^\varepsilon|^2\rho^\varepsilon\biggr)^{\!1/2}
\]
and
\begin{align*}
\iint |\beta-\beta^\varepsilon|^2\rho^\varepsilon
&= \iint |\beta|^2\rho^\varepsilon
- 2 \iint \beta\cdot\beta^\varepsilon\rho^\varepsilon
+ \iint |\beta^\varepsilon|^2\rho^\varepsilon \\
&= \iint ( |\beta|^2 \star_y g^\varepsilon ) \rho
- 2 \iint (\beta \star_y g^\varepsilon) \cdot \beta \rho
+ \iint |\beta^\varepsilon|^2\rho^\varepsilon
\to 0,
\end{align*}
when $\varepsilon \to 0$ by previous arguments and \cite[Lemma 3.1.1]{BKRSFPKEq}.
We also have
\begin{multline*}
\iint \nabla_y V(y,x) \cdot \bigl(\ell(x) - y\bigr)
\rho^\varepsilon(\dd x\dd y)
= \iint \Bigl( \nabla_y V(y,x) \cdot \bigl(\ell(x) - y\bigr) \Bigr)
\star_y g^\varepsilon \rho(\dd x\dd y) \\
\to
\iint \nabla_y V(y,x) \cdot \bigl(\ell(x) - y\bigr)\rho(\dd x\dd y)
\end{multline*}
when $\varepsilon \to 0$.
Finally, by Lemma~\ref{lem:grad-invariant} for the function
$\varphi(y) = y \cdot \nabla \Phi(y) - \Phi(y)$,
we have
\begin{align*}
\MoveEqLeft \iint\nabla_yV(y,x)\cdot\bigl(\ell(x)-y\bigr)\rho(\dd x\dd y) \\
&=\iint \ell(x)^\top \nabla^2\Phi(y)
\bigl(\ell(x) - y\bigr) \rho(\dd x\dd y) \\
&=\iint \bigl(\ell(x) - y\bigr)^\top \nabla^2\Phi(y)
\bigl(\ell(x) - y\bigr) \rho(\dd x\dd y) = I,
\end{align*}
where the last equality is exactly the definition \eqref{eq:def-I} of $I$.
Thus, we have shown
\[
\frac 12 \iint
\biggl| \nabla_x \log \frac{\rho^\varepsilon(x,y)}{\hat m_y(x)}
\biggr|^2 \rho^\varepsilon(\dd x\dd y)
\leqslant (D+2I) \lambda + o(1).
\]
Note that by the lower semicontinuity of (partial) Fisher information,
\[
\liminf_{\varepsilon \to 0}\iint
\biggl| \nabla_x \log \frac{\rho^\varepsilon(x,y)}{\hat m_y(x)}
\biggr|^2 \rho^\varepsilon(\dd x\dd y)
\geqslant
\iint
\biggl| \nabla_x \log \frac{\rho(x,y)}{\hat m_y(x)}
\biggr|^2 \rho(\dd x\dd y).
\]
We refer readers to \cite[Proof of Lemma~A.1]{uklpoc} for details.
Taking the limit $\varepsilon \to 0$, we obtain
\[
\frac 12\iint
\biggl| \nabla_x \log \frac{\rho(x,y)}{\hat m_y(x)}
\biggr|^2 \rho(\dd x\dd y)
\leqslant (D+2I) \lambda.
\]
Since $\rho^2$ is supported on $\mathcal K$,
for $\rho^2$-almost all $y \in \mathbb R^D$,
we have the following by the uniform LSI for $(\hat m_y)_{y \in \mathcal K}$:
\[
H(\rho | \hat \rho)
= \int H\bigl( \rho^{1|2}(\cdot|y) \big| \hat m_y \bigr) \rho^2(\dd y)
\leqslant \frac{C_\textnormal{LS}}{4} \!
\iint \biggl| \nabla_x \log \frac{\rho(x,y)}{\hat m_y(x)} \biggr|^2
\rho(\dd x\dd y),
\]
which completes the proof.
\end{proof}

\begin{rem}
If we formally integrate the static Fokker--Planck equation
\eqref{eq:fp-static} with $\log (\rho/\hat\rho)$ and integrate by parts,
we obtain
\begin{equation}
\label{eq:rho-rel-fisher-equality}
\frac 12 \iint
\biggl| \nabla_x \log \frac{\rho(x,y)}{\hat m_y(x)}
\biggr|^2 \rho(\dd x\dd y)
= \lambda D
+ 2 \lambda \iint \ell(x)^\top\nabla^2\Phi(y)\bigl(\ell(x)-y\bigr)
\rho(\dd x\dd y).
\end{equation}
However, the equality must not hold in all circumstances.
Indeed, if one artificially increases the dimension of $\ell$ and $\Phi$
by defining the new functions
\begin{align*}
\tilde \Phi (y_0, y_1) &= \Phi(y_0), \\
\tilde \ell (x) &= \bigl(\ell(x), 0\bigr),
\end{align*}
the right hand side of \eqref{eq:rho-rel-fisher-equality} increases
while the left hand side stays unchanged.
This phenomenon is caused by the fact that the equation \eqref{eq:fp-static}
is degenerate elliptic and lacks Laplacian in the $y$ directions.
To illustrate this effect, consider the first-order equation
\[
\nabla_y \cdot (y \rho) = 0
\]
in $d$ dimensions.
This equation has a probability solution $\rho = \delta_0$, the Dirac mass at
the origin.
Formally integrating the equation with $\log \rho$ and integrating by parts,
we have
\[
0 = \int \log \rho \nabla_y \cdot (y \rho)
= - \int \frac{\nabla_y \rho}{\rho} \cdot y \rho
= - \int \nabla_y \rho \cdot y
= \int \rho \nabla_y \cdot y
= \int \rho d = d,
\]
which is absurd.
\end{rem}

To complete the entropy estimate,
we provide in the following upper bounds for the integral $I$.

\begin{prop}
\label{prop:upper-bound-I}
Under the setting of the theorem,
the integral $I$ in \eqref{eq:def-I} satisfies the upper bound:
\begin{equation}
\label{eq:upper-bound-I-1}
I \leqslant 12 M_2
\Bigl( d C_\textnormal{LS} + 2 W_2^2 \bigl( \Law(X), m_* \bigr) \Bigr),
\end{equation}
where $X$ is the first component of the random variable
$(X,m)$ following the stationary distribution $P = P^\lambda$,
and $m_*$ is the unique invariant measure of the McKean--Vlasov process
\eqref{eq:mkv}.
If additionally $\nabla^2\Phi$ is convex and the quantity
\[
M_1 \coloneqq \sup_{x \in \mathbb R^d, y \in \mathbb R^D}
\ell(x)^\top \nabla^2\Phi(y) \ell(x)
\]
is finite, then we have the alternative upper bound:
\begin{equation}
\label{eq:upper-bound-I-2}
I \leqslant M_1.
\end{equation}
\end{prop}

\begin{proof}[Proof of Proposition~\ref{prop:upper-bound-I}]
First let us treat the simpler case where $M_1 < \infty$.
By applying Lemma~\ref{lem:grad-invariant} to the function
$\varphi(y) = y \cdot \nabla \Phi(y) - \Phi(y)$, we get
\begin{align*}
I &= \iint \bigl( \ell(x)^\top \nabla^2\Phi(y) \ell(x)
- y^\top \nabla^2\Phi(y)y \bigr) \rho(\dd x\dd y) \\
&\leqslant \iint \ell(x)^\top \nabla^2\Phi(y) \ell(x) \rho(\dd x\dd y)
\leqslant M_1.
\end{align*}
So the second claim of the proposition is proved.

Without the assumption $M_1 < \infty$, we note that,
for the second-order functional derivative
\[
\frac{\delta^2 F}{\delta m^2}(m,x',x'')
= \ell(x'')^\top \nabla^2\Phi(\langle \ell, m\rangle) \ell(x),
\]
we have
\[
I = \Expect \biggl[
\iint \frac{\delta^2 F}{\delta m^2}(m,x',x'')
(\delta_X - m)(\dd x') (\delta_X - m)(\dd x'') \biggr],
\]
where $(X,m)$ is a random variable following the stationary distribution $P$,
and $(X,\langle \ell, m\rangle)$ has the distribution $\rho$.
We observe
\begin{align*}
\bigl| D_m^2 F(m,x',x'') \bigr|
&= \biggl| \nabla_{x'} \nabla_{x''} \frac{\delta^2 F}{\delta m^2}
(m,x',x'') \biggr| \\
&= \bigl| \nabla\ell(x'')^\top \nabla^2\Phi(\langle\ell, m\rangle)
\nabla\ell(x') \bigr| \leqslant M_2.
\end{align*}
Then, by applying Lemma~\ref{lem:kantorovich-2-var} in
Appendix~\ref{app:kantorovich-2-var}
to a sequence of $\mathcal C^2$ functions approaching
$(x', x'') \mapsto \frac{\delta^2 F}{\delta m^2}(m,x',x'')$, we get
\begin{align*}
I &\leqslant M_2 \Expect \bigl[ W_2^2 (\delta_X, m) \bigr] \\
&= M_2 \Expect \biggl[ \int \lvert X - x' \rvert^2 m(\dd x') \biggr] \\
&\leqslant 2 M_2
\Expect \biggl[ \lvert X - \Expect[X] \rvert^2
+ \int \lvert x' - \Expect[X] \rvert^2 m(\dd x') \biggr].
\end{align*}
Let $\phi : \mathcal P_2(\mathbb R^d) \to \mathbb R$
be the functional defined by
\[
\phi(m) = \int \lvert x' - \Expect[X] \rvert^2 m(\dd x').
\]
We consider the sequence of $\mathcal C^1_\textnormal{b}$ ``soft cut-offs''
that approach $\phi$:
\[
\phi_n(m) = \sum_{i=1}^d \int n^2 \tanh^2
\biggl( \frac{x'^i - \Expect[X^i]}{n} \biggr) m(\dd x'),
\qquad\text{for}~n \in \mathbb N.
\]
Then, by applying the sequence $\phi_n$ to
Proposition~\ref{prop:si-stationary-pde-weak-sol} of stationary measure
and taking the limit $n \to \infty$, we get
\begin{multline*}
0 = \lambda \Expect\biggl[
\int \frac{\delta\phi}{\delta m}(m,x') (\delta_X - m)(\dd x') \biggr] \\
= \lambda \Expect\bigl[ \lvert X - \Expect[X] \lvert^2 \bigr]
- \lambda \Expect\biggr[\!\int \lvert x'
- \Expect[X] \rvert^2 m(\dd x') \biggr].
\end{multline*}
Thus, we have derived
\[
I \leqslant 4M_2 \Expect\bigl[ \lvert X - \Expect[X] \rvert^2 \bigr]
\eqqcolon 4M_2 \Var X,
\]
where $\Var X$ denotes the sum of the variances
of each component of the random vector $X$.
It remains only to find an upper bound for $\Var X$.
Note that, using the definition of Wasserstein distance
and the triangle inequality, and letting
$X_*$ be distributed as $m_*$, we get
\begin{align*}
\Var X &= W_2^2 \bigl( \Law(X), \delta_{\Expect[X]} \bigr) \\
&\leqslant 3 \Bigl(
W_2^2 \bigl( \Law(X), m_* \bigr)
+ W_2^2 \bigl( m_*, \delta_{\Expect[X_*]} \bigr)
+ W_2^2 \bigl(\delta_{\Expect[X_*]}, \delta_{\Expect[X]}\bigr) \Bigr) \\
&\leqslant
3 \Bigl( \Var X_* + 2 W_2^2\bigl(\Law(X), m_*\bigr) \Bigr),
\end{align*}
while the variance of $X_*$ is upper bounded by the Poincaré inequality:
\begin{align*}
\Var X_* &= \sum_{i=1}^d
\Biggl( \int \lvert x^i\rvert^2 m_*(\dd x)
- \biggl( \int x^i m_* (\dd x) \biggr)^{\!2} \Biggr) \\
&\leqslant C_\textnormal{LS}
\sum_{i=1}^d \int \lvert \nabla x^i\rvert^2 m_*(\dd x)
= C_\textnormal{LS} d.
\end{align*}
We then conclude by combining the three inequalities above.
\end{proof}

\subsection{Construction of another measure}

In this subsection, we construct another measure
in order to exploit the convexity of $\Phi$,
used for the proof of the first and second claims of the theorem.
Readers only interested in the last claim of the theorem
can now directly go to the next subsection.

Let $\mu = \mu^\lambda$ be the probability measure
on $\mathbb R^d \times \mathbb R^D$ characterized by the following formula:
\begin{multline}
\label{eq:def-mu}
\langle f, \mu \rangle
= \int f(x,y) \mu (\dd x \dd y) \\
= \Expect \biggl[ \int f(x, \langle \ell, m\rangle) m(\dd x) \biggr]
= \Expect \biggl[ \int f(x, Y) m(\dd x) \biggr]
\end{multline}
for all bounded and measurable
$f : \mathbb R^d \times \mathbb R^D \to \mathbb R$.
By taking $f$ depending only on the $y$ variable, we first realize that
the second marginals of $\rho$ and $\mu$ agree, that is,
\[\rho^2 = \mu^2.\]
In addition, we show the following important properties of $\mu$.

\begin{prop}
\label{prop:mu-properties}
Under the setting of the theorem,
for every $\mathcal C^2$ differentiable $\Psi : \mathbb R^{D} \to \mathbb R$
with bounded Hessian, we have
\[
\iint \nabla \Psi(y) \cdot \ell(x) (\mu - \rho) (\dd x \dd y) = 0.
\]
In particular, the respective first marginals $\mu^1$, $\rho^1$
of $\mu$, $\rho$ satisfy
\[
\int \ell(x) (\mu^1 - \rho^1) (\dd x) = 0.
\]
Moreover, denoting by $\mu^{1|2}(\cdot | \cdot)
: \mathcal B(\mathbb R^d) \times \mathbb R^D \to \mathbb R$
the conditional measure of $\mu$ given its second variable,
we have for $\rho^2$-almost all $y \in \mathbb R^{D}$,
\[
\bigl\langle \ell, \mu^{1|2} (\cdot | y) \bigr\rangle = y.
\]
\end{prop}

\begin{proof}[Proof of Proposition~\ref{prop:mu-properties}]
Consider the functional
\[
\phi(m) = \int f(x', \langle k, m\rangle) m(\dd x'),
\]
where $f \in \mathcal C^1_\textnormal{b}
(\mathbb R^d \times \mathbb R^D; \mathbb R)$
and $k \in \mathcal C_\textnormal{b}(\mathbb R^d; \mathbb R^D)$.
Then its linear functional derivative reads
\[
\frac{\delta \phi}{\delta m}(m,x')
= f(x', \langle k, m\rangle)
+ \int \nabla_y f(x'', \langle k, m\rangle) \cdot k(x') m(\dd x''),
\]
so $\phi$ belongs to the $\mathcal C^1_\textnormal{b}$ class.
Then, applying Proposition~\ref{prop:si-stationary-pde-weak-sol}
to the functional $\phi$, we get
\begin{align*}
0 &= \Expect \biggl[ \int \frac{\delta\phi}{\delta m}
(m,x') (\delta_X - m) (\dd x') \biggr] \\
&= \Expect [ f(X, m) ]
- \Expect\biggl[ \int f(x', \langle k, m\rangle) m(\dd x') \biggr] \\
&\hphantom{={}}\quad
+ \Expect \biggl[ \int \nabla_y f(x', \langle k, m\rangle)
\cdot \bigl(k (X) - \langle k,m\rangle\bigr) m(\dd x') \biggr]
\end{align*}
By approximation, the equality above holds for
$k = \ell$ and for all $\mathcal C^1$-continuous
$f$ with the following growth bounds:
\begin{align*}
\lvert f(x,y)\rvert &\leqslant M (1 + \lvert x\rvert^4 + \lvert y\rvert^2), \\
\lvert \nabla_y f(x,y) \rvert
&\leqslant M (1 + \lvert x\rvert^2 + \lvert y\rvert),
\end{align*}
that is to say, we have
\[
\langle f, \rho - \mu\rangle + \Expect \biggl[ \int \nabla_y f(x', Y)
\bigl( \ell(X) - Y \bigr) m(\dd x') \biggr] = 0,
\]
where, as before, $Y = \langle\ell, m\rangle$.
Specializing to $f(x,y) = \nabla \Psi(y) \cdot \ell(x)$, we obtain
\begin{align*}
\langle f, \mu - \rho\rangle
&= \Expect \biggl[ \int \ell(x)^\top
\nabla^2\Psi(Y) \bigl(\ell(X) - Y\bigr)m(\dd x)\biggr] \\
&= \Expect \bigl[ Y^\top
\nabla^2\Psi(Y) \bigl(\ell(X) - Y\bigr)\bigr]=0,
\end{align*}
where the last equality is due to Lemma \ref{lem:grad-invariant},
as for $\varphi(y) \coloneqq \nabla \Psi(y) \cdot y - \Psi(y)$,
we have $\nabla \varphi(y) = \nabla^2\Psi(y)y$.
So the first claim is proved.
Taking $\Psi(y) = y^\nu$, for $\nu = 1,\ldots,D$, yields
the second claim.

For the last claim,
we take $f(x,y) = \ell(x)  g(y)$
for $g : \mathbb R^{D} \to \mathbb R$ of linear growth
in the defining equation \eqref{eq:def-mu} of $\mu$.
Then, we get
\begin{multline*}
\int g(y) \biggl(\int \ell(x) \mu^{1|2} (\dd x|y)\biggr) \mu^2(\dd y)
= \iint f(x,y) \mu(\dd x\dd y)
= \Expect [ Y g(Y) ] \\
= \int g(y) y \mu^2(\dd y).
\end{multline*}
The desired property follows from the arbitrariness of $g$.
\end{proof}

\subsection{Proving the theorem}

Having established the entropy estimate
and constructed the auxiliary measure,
we finally move to the central part of the proof,
which consists of six steps.
The aim of the first five steps is to show
the first and the second claims of the theorem,
and in the last step we prove the last claim.

\proofstep{Step 1: Control of the symmetrized entropy}
We aim at controlling the symmetrized entropy
\[
\int \bigl( H(\hat m_y | m_*) + H(m_* | \hat m_y) \bigr) \rho^2(\dd y)
\]
in this step. First observe
\begin{multline}
\label{eq:sys_entropy}
\int \bigl( H(\hat m_y | m_*) + H(m_* | \hat m_y) \bigr) \rho^2(\dd y) \\
= 2 \iint \bigl(V(y_*,x) - V(y,x)\bigr)
\bigl( \hat m_y(\dd x) - m_*(\dd x) \bigr) \rho^2(\dd y).
\end{multline}
In order to control the right hand side above,
we turn to the probability measure $\mu$ introduced in \eqref{eq:def-mu}.
Recall that $m_*$ is the invariant measure of the McKean--Vlasov
\eqref{eq:mkv},
and $y_* \coloneqq \langle\ell, m_*\rangle$.
The convexity of $\Phi$ implies the convexity of $F$ as a functional,
and as a result, for $\rho^2$-almost all $y \in \mathbb R^{D}$, we have
the tangent inequalities
\begin{multline}
\label{eq:tangent}
\int V\bigl(\bigl\langle\ell, \mu^{1|2}(\cdot|y)\bigr\rangle, x\bigr)
\bigl(\mu^{1|2}(\dd x | y) - m_*(\dd x)\bigr) \\
\geqslant F\bigl(\mu^{1|2}(\cdot|y)\bigl) - F(m_*) \\
\geqslant \int V(y_*, x) \bigl(\mu^{1|2}(\dd x | y) - m_*(\dd x)\bigr).
\end{multline}
Thanks to the last claim of Proposition~\ref{prop:mu-properties},
the leftmost term satisfies, for $\rho^2$-almost all $y \in \mathbb R^D$,
\[
\int V\bigl(\bigl\langle\ell, \mu^{1|2}(\cdot|y)\bigr\rangle, x\bigr)
\bigl(\mu^{1|2}(\dd x | y) - m_*(\dd x)\bigr)
= \int V(y, x) \bigl(\mu^{1|2}(\dd x | y) - m_*(\dd x)\bigr).
\]
Hence, integrating the tangent inequalities \eqref{eq:tangent} above
by $\rho^2$, we get
\begin{multline}
\label{eq:convergence-energy-ineq-pre}
\iint V(y, x) \bigl(\mu^{1|2}(\dd x | y) - m_*(\dd x)\bigr) \rho^2(\dd y) \\
\geqslant \iint V(y_*, x) \bigl(\mu^{1|2}(\dd x | y) - m_*(\dd x)\bigr)
\rho^2(\dd y).
\end{multline}
Using $\mu^2 = \rho^2$ and applying Proposition~\ref{prop:mu-properties} to
$V(y,x) = \nabla \Phi(y) \cdot \ell(x)$, we know that
the left hand side of \eqref{eq:convergence-energy-ineq-pre} satisfies
\begin{align*}
\MoveEqLeft \iint V(y, x) \bigl(\mu^{1|2}(\dd x | y)
- m_*(\dd x)\bigr) \rho^2(\dd y) \\
&=\iint V(y,x) \mu(\dd x\dd y) - \iint V(y,x) m_*(\dd x) \rho^2(\dd y) \\
&=\iint V(y,x) \rho(\dd x\dd y) - \iint V(y,x) m_*(\dd x) \rho^2(\dd y) \\
&= \iint V(y, x) \bigl(\rho^{1|2}(\dd x | y) - m_*(\dd x)\bigr) \rho^2(\dd y).
\end{align*}
The right hand side of \eqref{eq:convergence-energy-ineq-pre} satisfies
\begin{align*}
\MoveEqLeft \iint V(y_*, x) \bigl(\mu^{1|2}(\dd x | y) - m_*(\dd x)\bigr)
\rho^2(\dd y) \\
&= \iint V(y_*, x) \mu(\dd x\dd y) - \iint V(y_*, x) m_*(\dd x) \\
&= \nabla \Phi(y_*) \cdot \int \ell(x)
\bigl(\mu^{1}(\dd x) - m_*(\dd x)\bigr) \\
&= \nabla \Phi(y_*) \cdot \int \ell(x)
\bigl(\rho^{1}(\dd x) - m_*(\dd x)\bigr) \\
&= \iint V(y_*, x) \bigl(\rho^{1|2}(\dd x | y) - m_*(\dd x)\bigr) \rho^2(\dd y),
\end{align*}
where the third equality is due to the last claim
of Proposition~\ref{prop:mu-properties}.
Thus, we have derived
\begin{equation}
\label{eq:convergence-energy-ineq}
\iint \bigl(V(y, x) - V(y_*,x)\bigr)
\bigl(\rho^{1|2}(\dd x | y) - m_*(\dd x)\bigr) \rho^2(\dd y)
\geqslant 0.
\end{equation}
Therefore, to dominate the right hand side of \eqref{eq:sys_entropy},
it remains to control the following term.
Using the Kantorovich duality, we get
\begin{align*}
\MoveEqLeft
\biggl|\iint \bigl( V(y,x) - V(y_*,x) \bigr)
\bigl( \rho^{1|2}(\dd x|y) - \hat m_y(\dd x) \bigr)
\rho^2(\dd y)\biggr| \\
&= \biggl|
\int \bigl(\nabla\Phi(y) - \nabla\Phi(y_*)\bigr)
\cdot \bigl\langle \ell, \rho^{1|2}(\cdot|y) - \hat m_y \bigr\rangle
\rho^2(\dd y) \biggr| \\
&= \biggl|
\iint_0^1
(y - y_*)^\top \nabla^2\Phi\bigl((1-t)y + ty_*\bigr)
\bigl\langle \ell, \rho^{1|2}(\cdot|y) - \hat m_y \bigr\rangle
\dd t\,\rho^2(\dd y)
\biggr| \\
&\leqslant
\biggl( \iint_0^1 (y - y_*)^\top
\nabla^2\Phi\bigl((1-t)y + ty_*\bigr)(y-y_*) \dd t\,\rho^2(\dd y)
\biggr)^{\!1/2} \\
&\hphantom{\leqslant{}}\quad\times\biggl( \iint_0^1
\Bigl|\Bigl\langle \nabla^2 \Phi \bigl((1-t)y+ty_*\bigr)^{1/2}\ell,
\rho^{1|2}(\cdot|y) - \hat m_y\Bigr\rangle\Bigr|^2
\dd t\,\rho^2(\dd y) \biggr)^{\!1/2} \\
&\leqslant
\biggl( \iint_0^1 (y - y_*)^\top
\nabla^2\Phi\bigl((1-t)y + ty_*\bigr)(y-y_*) \dd t\,\rho^2(\dd y)
\biggr)^{\!1/2} \\
&\hphantom{\leqslant{}}\quad\times
\sqrt{M_2} \biggl( \int W_1^2\bigl( \rho^{1|2}(\cdot|y), \hat m_y \bigr)
\rho^2(\dd y) \biggr)^{\!1/2} \\
&\eqqcolon \sqrt{M_2 v(\rho^2)}
\biggl( \int W_1^2\bigl( \rho^{1|2}(\cdot|y), \hat m_y \bigr)
\rho^2(\dd y) \biggr)^{\!1/2},
\end{align*}
where, by Assumption~\ref{assu:M2},
$\bigl|\nabla^2\Phi(y')^{1/2}\nabla\ell(x)\bigr| \leqslant \sqrt{M_2}$
for all $x \in \mathbb R^d$, $y' \in \mathbb R^D$,
and $v(\rho^2)$ is the quantity to be controlled
in the first claim of the theorem.
The uniform LSI for $(\hat m_y)_{y \in \mathcal K}$
implies a uniform Talagrand's transport inequality, from which we obtain
\begin{align*}
\int W_1^2\bigl(\rho^{1|2}(\cdot|y), \hat m_y\bigr)\rho^2(\dd y)
&\leqslant \int W_2^2\bigl(\rho^{1|2}(\cdot|y), \hat m_y\bigr)\rho^2(\dd y) \\
&\leqslant C_\textnormal{LS}
\int H\bigl(\rho^{1|2}(\cdot|y)\big|\hat m_y\bigr)\rho^2(\dd y)
= C_\textnormal{LS} H(\rho | \hat \rho),
\end{align*}
as $\rho^2$ is supported on $\mathcal K$.
Combining the two inequalities above with \eqref{eq:sys_entropy} and
\eqref{eq:convergence-energy-ineq}, we obtain
\begin{equation}
\label{eq:bound-sym-entropy}
\int \bigl( H(\hat m_y | m_*) + H(m_* | \hat m_y) \bigr) \rho^2(\dd y)
\leqslant
2 \sqrt{M_2C_\textnormal{LS} v(\rho^2) H(\rho | \hat \rho)}.
\end{equation}

\proofstep{Step 2: Control of the conditional Wasserstein distance}
Now, using again the Talagrand's transport inequality
for $\hat m_y$ and $m_*$ (note that $m_* = \hat m_{y_*}$
for $y_* = \langle \ell, m_*\rangle$), we get for $\rho^2$-almost all
$y \in \mathbb R^D$,
\[
W_2^2(\hat m_y, m_*) \leqslant \frac{C_\textnormal{LS}}{2}
\bigl( H(\hat m_y | m_*) + H(m_* | \hat m_y) \bigr),
\]
while the triangle inequality and the transport inequality imply
\begin{align*}
W_2^2 \bigl( \rho^{1|2} (\cdot|y), m_* \bigr)
&\leqslant 2 \Bigl(
W_2^2 \bigl( \rho^{1|2} (\cdot|y), \hat m_y \bigr)
+ W_2^2 (\hat m_y, m_*) \Bigr) \\
&\leqslant 2 C_\textnormal{LS}
H \bigl( \rho^{1|2}(\cdot|y), \hat m_y \bigr)
+ 2 W_2^2(\hat m_y, m_*).
\end{align*}
So, combining the three inequalities above
and integrating with $\rho^2$, we find
\begin{equation}
\label{eq:bound-w2-pre}
\int W_2^2 \bigl( \rho^{1|2} (\cdot|y), m_* \bigr) \rho^2(dy)
\leqslant 2 C_\textnormal{LS} H(\rho | \hat \rho)
+ 2 \sqrt{M_2C_\textnormal{LS}^3 v(\rho^2) H(\rho | \hat \rho)}.
\end{equation}

\proofstep{Step 3: Control of $v(\rho^2)$ by $H(\rho|\hat\rho)$}
Applying Proposition~\ref{prop:mu-properties}
to the function $\Psi(y) = \Phi(y) - \nabla \Phi(y_*) \cdot y$,
where, as we recall, $y_* = \langle \ell, m_*\rangle$, we get
\begin{align*}
0 &= \iint \nabla \Psi(y) \cdot \ell(x) (\mu - \rho) (\dd x\dd y) \\
&= \iint \bigl(\nabla\Phi(y) - \nabla\Phi(y_*)\bigr)
\cdot \ell(x)(\mu-\rho)(\dd x\dd y) \\
&= \iint_0^1 (y-y_*)^\top\nabla^2\Phi\bigl((1-t)y+ty_*\bigr)
\bigl\langle\ell, \mu^{1|2}(\cdot|y) - \rho^{1|2}(\cdot|y)\bigr\rangle
\dd t\,\rho^2(\dd y) \\
&= \iint_0^1 (y-y_*)^\top\nabla^2\Phi\bigl((1-t)y+ty_*\bigr)
\bigl( y - \bigl\langle\ell,\rho^{1|2}(\cdot|y)\bigr\rangle
\bigr)\dd t\,\rho^2(\dd y),
\end{align*}
where the last equality is due to the last claim of
Proposition~\ref{prop:mu-properties}.
In other words, we have
\begin{multline*}
\iint_0^1 (y-y_*)^\top\nabla^2\Phi\bigl((1-t)y+ty_*\bigr)
(y-y_*)\dd t\,\rho^2(\dd y) \\
= \iint_0^1 (y - y_*)^\top
\nabla^2 \Phi\bigl((1-t)y+ty_*\bigr)
\bigl\langle\ell, \rho^{1|2}(\cdot|y) - m_*\bigr\rangle\dd t\,\rho^2(\dd y),
\end{multline*}
and this implies, by Cauchy--Schwarz,
\begin{align*}
\MoveEqLeft \iint_0^1 (y-y_*)^\top\nabla^2\Phi\bigl((1-t)y+ty_*\bigr)
(y-y_*)\dd t\,\rho^2(\dd y) \\
&\leqslant
\iint_0^1 \bigl\langle \ell, \rho^{1|2}(\cdot|y) - m_*\bigr\rangle^\top
\nabla^2\Phi\bigl((1-t)y+ty_*\bigr)
\bigl\langle \ell, \rho^{1|2}(\cdot|y) - m_*\bigr\rangle
\dd t\,\rho^2(\dd y) \\
&= \iint_0^1
\Bigl|\Bigl\langle \nabla^2 \Phi \bigl((1-t)y+ty_*\bigr)^{1/2}\ell,
\rho^{1|2}(\cdot|y) - m_*\Bigr\rangle\Bigr|^2
\dd t\,\rho^2(\dd y).
\end{align*}
As $\bigl|\nabla^2\Phi(y')^{1/2}
\nabla \ell(x)\bigr| \leqslant \sqrt{M_2}$
for all $y' \in \mathbb R^D$ and $x \in \mathbb R^d$,
we have, by the Kantorovich duality,
\begin{multline}
\label{eq:bound-rho2}
v(\rho^2) =
\iint_0^1 (y-y_*)^\top\nabla^2\Phi\bigl((1-t)y+ty_*\bigr)
(y-y_*)\dd t\,\rho^2(\dd y) \\
\leqslant M_2
\int W_1^2\bigl(\rho^{1|2}(\cdot|y), m_*\bigr) \rho^2(\dd y).
\end{multline}
Thus, using the fact that $W_1 \leqslant W_2$ and the inequality
\eqref{eq:bound-w2-pre}, we obtain
\[
v(\rho^2) \leqslant
2 M_2 C_\textnormal{LS} H(\rho | \hat \rho)
+ 2 \sqrt{M_2^3 C_\textnormal{LS}^3 v(\rho^2) H(\rho | \hat \rho)}.
\]
Introduce the ``adimensionalized'' variable
\[
\upsilon = \frac{v(\rho^2)}
{4M_2^3 C_\textnormal{LS}^3 H(\rho|\hat\rho)}.
\]
Then the inequality above reads
\[
\upsilon \leqslant \frac{1}{2M_2^2C_\textnormal{LS}^2}
+ \sqrt{\upsilon}
\leqslant \frac{1}{2M_2^2C_\textnormal{LS}^2}
+ \frac 12 + \frac 12 \upsilon.
\]
Hence, we get
\begin{equation}
\label{eq:upper-bound-v}
v(\rho^2) \leqslant
4 M_2 C_\textnormal{LS}
\bigl(M_2^2C_\textnormal{LS}^2 + 1\bigr) H(\rho | \hat \rho).
\end{equation}

\proofstep{Step 4:
Control of Wasserstein and TV distances by $H(\rho|\hat\rho)$}
By inserting \eqref{eq:upper-bound-v}
into \eqref{eq:bound-w2-pre}, and noting,
by the definition of the Wasserstein distance,
\[
W_2^2 (\rho^1, m_*) \leqslant
\int W_2^2 \bigl( \rho^{1|2}(\cdot | y), m_* \bigr) \rho^2(\dd y),
\]
we get
\begin{equation}
\label{eq:upper-bound-w}
W_2^2 (\rho^1, m_*)
\leqslant \Bigl( 2C_\textnormal{LS} + 4M_2 C_\textnormal{LS}^2
\bigl(M_2^2C_\textnormal{LS}^2 + 1\bigr)^{1/2} \Bigr) H(\rho|\hat \rho).
\end{equation}
For the total variation distance,
we observe that the Csiszár--Kullback--Pinsker inequality implies
\begin{align*}
\int \bigl\lVert \rho^{1|2}(\cdot | y) - \hat m_y \bigr\rVert_\textnormal{TV}^2
\rho^2(\dd y)
&\leqslant 2 H(\rho | \hat \rho), \\
\int \lVert \hat m_y - m_* \rVert_\textnormal{TV}^2
\rho^2(\dd y)
&\leqslant \int \bigl( H(\hat m_y | m_*) + H(m_* | \hat m_y) \bigr)
\rho^2(\dd y).
\end{align*}
By the triangle and Jensen's inequalities, we get
\begin{align*}
\lVert \rho^1 - m_* \rVert_\textnormal{TV}^2
&\leqslant \int \bigl\lVert \rho^{1|2}(\cdot | y)
- m_* \bigr\rVert_\textnormal{TV}^2 \rho^2(\dd y) \\
&\leqslant 2 \int \Bigl(
\bigl\lVert \rho^{1|2}(\cdot|y) - \hat m_y\bigr\rVert_\textnormal{TV}^2
+ \lVert \hat m_y - m_* \rVert_\textnormal{TV}^2 \Bigr) \rho^2(\dd y) \\
&\leqslant 4 H(\rho|\hat\rho)
+ 2 \int \bigl( H(\hat m_y | m_*) + H(m_* | \hat m_y)\bigr) \rho^2(\dd y).
\end{align*}
Then, by inserting \eqref{eq:upper-bound-v}
into \eqref{eq:bound-sym-entropy}, we get
\begin{equation}
\label{eq:upper-bound-TV}
\lVert \rho^1 - m_* \rVert_\textnormal{TV}^2
\leqslant \Bigl( 4 + 8M_2 C_\textnormal{LS}
\bigl(M_2^2C_\textnormal{LS}^2 + 1\bigr)^{1/2} \Bigr) H(\rho|\hat \rho).
\end{equation}

\proofstep{Step 5: Control of $H(\rho|\hat \rho)$ and conclusion
for the convex case}
In the case where
\[
M_1 \coloneqq \sup_{x \in \mathbb R^d, y \in \mathbb R^D}
\ell(x)^\top \nabla^2\Phi(y) \ell(x) < \infty,
\]
by Proposition~\ref{prop:rho-hat-rho-entropy} and
\eqref{eq:upper-bound-I-2} in Proposition~\ref{prop:upper-bound-I},
we immediately get
\[
H(\rho | \hat \rho) \leqslant
\frac{C_\textnormal{LS}}{2} (D + 2M_1) \lambda = H'.
\]

In the general case where $M_1$ is not necessarily finite,
Proposition \ref{prop:rho-hat-rho-entropy}
and \eqref{eq:upper-bound-I-1} in Proposition~\ref{prop:upper-bound-I} yield
\[
H(\rho | \hat \rho)
\leqslant
\frac{C_\textnormal{LS}}{2}
( D + 24 M_2 C_\textnormal{LS} d) \lambda
+ 24 M_2 C_\textnormal{LS} W_2^2(\rho^1, m_*) \lambda.
\]
Together with the upper bound \eqref{eq:upper-bound-w} of $W_2^2(\rho^1, m_*)$,
we get
\[
H(\rho|\hat \rho)
\leqslant
\frac{C_\textnormal{LS}(D + 24 M_2C_\textnormal{LS}d)\lambda}
{2 - 96 M_2 C_\textnormal{LS}^2
\Bigl( 1 + 2 M_2C_\textnormal{LS}
\bigl( M_2^2C_\textnormal{LS}^2+1\bigr)^{1/2} \Bigr) \lambda}
= H,
\]
for
\[
\lambda <
\frac{1}{48M_2 C_\textnormal{LS}^2
\Bigl( 1 + 2 M_2C_\textnormal{LS}
\bigl( M_2^2C_\textnormal{LS}^2+1\bigr)^{1/2} \Bigr)}
= \lambda_0.
\]
We obtain the desired estimates on $v(\rho^2)$, Wasserstein and TV distances,
by inserting the upper bounds of $H(\rho|\hat\rho)$
for the respective cases into
\eqref{eq:upper-bound-v}, \eqref{eq:upper-bound-w}, \eqref{eq:upper-bound-TV}.

\medskip
Now we work with a concave $\Phi$ and
prove the last claim of the theorem.

\proofstep{Step 6: Case of concave $\Phi$}
Observe first that the mapping $y \mapsto \nabla^2\Phi(y) \cdot \hat y$
is a gradient:
\begin{align*}
\nabla^2\Phi(y) \cdot \hat y &= \nabla^2\Phi(y) \cdot
\frac{\int \ell(x)\exp\bigl( -2\nabla\Phi(y) \cdot \ell(x)\bigr)\dd x}
{\int \exp\bigl( -2\nabla \Phi(y) \cdot \ell(x)\bigr)\dd x} \\
&= - \frac 12\nabla_y
\biggl(\log \int \exp \bigl( -2\nabla\Phi(y) \cdot \ell(x) \bigr) \dd x\biggr).
\end{align*}
This identity is analogous to the fact in thermodynamics
that when we derive the free energy with respect to a variable,
we get the statistical average of its response variable.
Moreover,
\[
\nabla^2\Phi(y)\cdot y=
\nabla_y \bigl( \nabla \Phi(y) \cdot y - \Phi(y) \bigr).
\]
Thus, taking the test function
\[
\varphi(y)
= - \frac 12 \log \int \exp \bigl( -2\nabla\Phi(y) \cdot \ell(x) \bigr) \dd x
- \nabla \Phi(y) \cdot y + \Phi(y)
\]
in Lemma~\ref{lem:grad-invariant}, we get
\[
\iint \bigl( \ell(x) - y \bigr)^\top
\nabla^2\Phi(y) ( \hat y - y) \rho(\dd x\dd y) = 0.
\]
Consequently,
\begin{align*}
\MoveEqLeft
-\iint (\hat y-y)^\top \nabla^2\Phi(y) (\hat y-y) \rho(\dd x\dd y) \\
&= - \iint \bigl(\hat y - \ell(x)\bigr)^\top
\nabla^2\Phi(y) (\hat y-y) \rho(\dd x\dd y) \\
&= \int \Bigl< \bigl(-\nabla^2\Phi(y)\bigr)^{1/2}
\ell, \hat m_y - \rho^{1|2}(\cdot|y) \Bigr>
\cdot \bigl(-\nabla^2\Phi(y)\bigr)^{1/2}
(\hat y - y) \rho^2(\dd y) \\
&\leqslant
\biggl(- \int ( \hat y-y)^\top
\nabla^2\Phi(y)
( \hat y-y) \rho^2(\dd y)\biggr)^{\!1/2}\\
&\hphantom{\leqslant{}}\quad\times
\biggl( \int \Bigl|\Bigl< \bigl(-\nabla^2\Phi(y)\bigr)^{1/2}
\ell, \hat m_y - \rho^{1|2}(\cdot|y) \Bigr>\Bigr|^2
 \rho^2(\dd y)\biggr)^{\!1/2},
\end{align*}
which implies that
\[
-\int (\hat y-y)^\top \nabla^2\Phi(y) (\hat y-y) \rho^2(\dd y)
\leqslant\int \Bigl|\Bigl< \bigl(-\nabla^2\Phi(y)\bigr)^{1/2}
\ell, \hat m_y - \rho^{1|2}(\cdot|y) \Bigr>\Bigr|^2
 \rho^2(\dd y).
\]
The term on the right satisfies
\begin{align*}
\int\Bigl|\Bigl< \bigl(-\nabla^2\Phi(y)\bigr)^{1/2}
\ell, \hat m_y - \rho^{1|2}(\cdot|y) \Bigr>\Bigr|^2
 \rho^2(\dd y)
&\leqslant M_2 \int W_1^2\bigl( \hat m_y, \rho^{1|2}(\cdot|y)\bigr)
\rho^2(\dd y) \\
&\leqslant M_2C_\textnormal{LS}
\int H\bigl(\rho^{1|2}(\cdot|y)\big|\hat m_y\bigr) \rho^2(\dd y)
\end{align*}
by the uniform LSI for $(\hat m_y)_{y \in \mathcal K}$.
The entropy estimate in Proposition~\ref{prop:rho-hat-rho-entropy} gives
\[
\int H\bigl(\rho^{1|2}(\cdot|y)\big|\hat m_y\bigr) \rho^2(\dd y)
\leqslant \frac{C_\textnormal{LS}}{2} (D + 2I) \lambda
\leqslant \frac{C_\textnormal{LS}D\lambda}{2},
\]
as in the case of concave $\Phi$,
the term $I \leqslant 0$ by its definition \eqref{eq:def-I}.
\qed

%% file: simcmc-app.tex
\section{A lemma on Wasserstein duality}
\label{app:kantorovich-2-var}

\begin{lem}
\label{lem:kantorovich-2-var}
Let $f : \mathbb R^{d_1} \times \mathbb R^{d_2} \to \mathbb R$
be a $\mathcal C^2$ function such that
the Euclidean operator norm of $\nabla_x\nabla_yf(x,y)$ satisfies
\[
\lvert \nabla_x \nabla_y f(x,y) \rvert \leqslant M
\]
for all $x \in \mathbb R^{d_1}$ and $y \in \mathbb R^{d_2}$.
Then, for all $\mu$, $\mu' \in \mathcal P_2(\mathbb R^{d_1})$
and $\nu$, $\nu' \in \mathcal P_2(\mathbb R^{d_2})$, we have
\[
\biggl|\iint f(x,y) (\mu - \mu')(\dd x) (\nu - \nu')(\dd y)\biggr|
\leqslant M W_2(\mu, \mu') W_2(\nu, \nu').
\]
\end{lem}

\begin{proof}[Proof of Lemma~\ref{lem:kantorovich-2-var}]
Let $\pi \in \Pi(\mu, \mu')$ be the $W_2$-optimal transport plan
between $\mu$ and $\mu'$,
and $\pi' \in \Pi(\nu, \nu')$ be that between $\nu$ and $\nu'$.
Construct the random variable $\bigl((X, X'), (Y, Y')\bigr)$
distributed as $\pi \otimes \pi'$.
Then, we have
\begin{align*}
\MoveEqLeft \iint f(x,y) (\mu - \mu')(\dd x) (\nu - \nu')(\dd y) \\
&= \Expect [ f(X,Y) - f(X', Y) - f(X,Y') + f(X',Y')] \\
&= \Expect \biggl[ \iint_{[0,1]^2} (X-X')^\top \\
&\hphantom{{}= \Expect \biggl[ \iint_{[0,1]^2}}\quad
\nabla_x \nabla_y f \bigl( (1-t)X + tX', (1-s)Y + sY' \bigr)
(Y-Y') \dd t \dd s \biggr].
\end{align*}
Therefore, taking absolute values, we get
\begin{align*}
\MoveEqLeft \biggl|\iint f(x,y)
(\mu - \mu')(\dd x) (\nu - \nu')(\dd y)\biggr| \\
&\leqslant M \Expect \bigl[ \lvert X - X'\rvert \lvert Y - Y' \rvert \bigr] \\
&\leqslant M \Expect \bigl[ \lvert X - X'\rvert^2 \bigr]^{1/2}
\Expect \bigl[ \lvert Y - Y' \rvert^2 \bigr]^{1/2} \\
&= M W_2(\mu, \mu') W_2(\nu, \nu'),
\end{align*}
which concludes the proof.
\end{proof}

\section{Algorithm}
\label{sec:algorithm}

\begin{algorithm}[H]
\caption{Training two-layer neural network by self-interacting diffusions}
\SetKwInOut{Input}{Input}
\SetKwInOut{Output}{Output}

\Input{Activation func.\ $\varphi$,
truncation func.\ $\tau$,
dataset $\{z_k,y_k\}^K_{k=1}$,
volatility $\sigma$,
regularization constant $\gamma$,
initial distribution $m_0$,
time step $\delta t$,
time horizon $T$,
non-increasing piecewise constant func.\ %
$\lambda:[0,T]\to \mathbb (0,\infty)$}
\Output{$Y_T$}
generate $X_0=(C_0,A_0,B_0)\sim m_0$; $Y_0=l(X_0)$\;
\For{$t=0$, $\delta t$, $2\delta t$, \dots, $T-\delta t$}{
generate i.i.d.\ $\mathcal{N}_t\sim \mathcal{N}(0,1)$\;
$X_{t+\delta t}\leftarrow X_t
- \big(\frac{1}{K}\sum_{k=1}^K Y_t^k\nabla f(X_t,z_k)
+\sigma^2\gamma X_t\big)\delta t +\sigma\sqrt{\delta t}\mathcal{N}_t$\;
\If{$\lambda(t)\equiv \lambda_*$ is constant function}
{$Y_{t+\delta t}\leftarrow (1-\lambda_* \delta t)Y_t
+ \lambda_* \delta t \,\ell(X_{t+\delta t})$\;
\tcp{Update corresponding to fixed value of $\lambda$}}
\Else {$Y_{t+\delta t}\leftarrow \bigl(1-\lambda(t) \delta t\bigr)Y_t
+ \lambda(t) \delta t \,\ell(X_{t+\delta t})$\;
\tcp{Update corresponding to annealing $\lambda(t)
=\lambda_\textnormal{a}(t)$}}
}

\end{algorithm}

%% file: my.bib
@article{uklpoc,
	title={Uniform-in-time propagation of chaos for kinetic mean field
{Langevin} dynamics},
	author={Fan Chen and Yiqing Lin and Zhenjie Ren and Songbo Wang},
	fjournal={Electronic Journal of Probability},
	shortjournal={Electron.\ J.\ Probab.},
	journal={Electron.\ J.\ Probab.},
	year={2024},
	volume={29},
	pagetotal={43},
	eid={17},
	doi={10.1214/24-EJP1079}
}

@article{ulpoc,
	title={Uniform-in-time propagation of chaos for mean field {Langevin} dynamics},
	author={Fan Chen and Zhenjie Ren and Songbo Wang},
	year={2024},
	journal={Ann.\ Inst.\ Henri Poincaré, Probab.\ Stat.},
	eprint={2212.03050},
	archivePrefix={arXiv},
	primaryClass={math.PR},
	pubstate={forthcoming}
}

@article{efp,
	author={Fan Chen and Zhenjie Ren and Songbo Wang},
	title={Entropic Fictitious Play for Mean Field Optimization Problem},
	fjournal={Journal of Machine Learning Research},
	shortjournal={J.\ Mach.\ Learn.\ Res.},
	journal={J.\ Mach.\ Learn.\ Res.},
	year={2023},
	volume={24},
	eid={211},
	pagetotal={36},
	url={http://jmlr.org/papers/v24/22-0411.html}
}


%% file: ref.bib
@Article{LLFSharp,
 Author = {Lacker, Daniel and Le Flem, Luc},
 Title = {Sharp uniform-in-time propagation of chaos},
 FJournal = {Probability Theory and Related Fields},
 Journal = {Probab. Theory Relat. Fields},
 ISSN = {0178-8051},
 Volume = {187},
 Number = {1-2},
 Pages = {443--480},
 Year = {2023},
 Language = {English},
 DOI = {10.1007/s00440-023-01192-x},
 zbMATH = {7735855}
}

@Article{HolleyStroockLSI,
 Author = {Holley, Richard and Stroock, Daniel},
 Title = {Logarithmic {Sobolev} inequalities and stochastic {Ising} models},
 FJournal = {Journal of Statistical Physics},
 Journal = {J. Stat. Phys.},
 ISSN = {0022-4715},
 Volume = {46},
 Number = {5-6},
 Pages = {1159--1194},
 Year = {1987},
 Language = {English},
 DOI = {10.1007/BF01011161},
 zbMATH = {4117639},
 Zbl = {0682.60109}
}

@Article{EberleReflectionCoupling,
 Author = {Eberle, Andreas},
 Title = {Reflection couplings and contraction rates for diffusions},
 FJournal = {Probability Theory and Related Fields},
 Journal = {Probab. Theory Relat. Fields},
 ISSN = {0178-8051},
 Volume = {166},
 Number = {3-4},
 Pages = {851--886},
 Year = {2016},
 Language = {English},
 DOI = {10.1007/s00440-015-0673-1},
 zbMATH = {6657178},
 Zbl = {1367.60099}
}

@article{EGZCoupling,
 Author = {Eberle, Andreas and Guillin, Arnaud and Zimmer, Raphael},
 Title = {Couplings and quantitative contraction rates for {Langevin} dynamics},
 FJournal = {The Annals of Probability},
 Journal = {Ann. Probab.},
 ISSN = {0091-1798},
 Volume = {47},
 Number = {4},
 Pages = {1982--2010},
 Year = {2019},
 Language = {English},
 DOI = {10.1214/18-AOP1299},
 zbMATH = {7114709},
 Zbl = {1466.60160}
}

@Article{HRSS,
 Author = {Hu, Kaitong and Ren, Zhenjie and {\v{S}}i{\v{s}}ka, David and Szpruch, {\L}ukasz},
 Title = {Mean-field {Langevin} dynamics and energy landscape of neural networks},
 FJournal = {Annales de l'Institut Henri Poincar{\'e}. Probabilit{\'e}s et Statistiques},
 Journal = {Ann. Inst. Henri Poincar{\'e}, Probab. Stat.},
 ISSN = {0246-0203},
 Volume = {57},
 Number = {4},
 Pages = {2043--2065},
 Year = {2021},
 Language = {English},
 DOI = {10.1214/20-AIHP1140},
 zbMATH = {7481278},
 Zbl = {1492.65023}
}

@Book{AGSGradientFlows,
 Author = {Ambrosio, Luigi and Gigli, Nicola and Savar{\'e}, Giuseppe},
 Title = {Gradient flows in metric spaces and in the space of probability measures},
 Edition = {2nd ed.},
 ISBN = {978-3-7643-8721-1},
 Year = {2008},
 Publisher = {Birkh{\"a}user},
 Address = {Basel},
 Language = {English},
 zbMATH = {5233008},
 Zbl = {1145.35001}
}

@article{MMNMF,
 Author = {Mei, Song and Montanari, Andrea and Nguyen, Phan-Minh},
 Title = {A mean field view of the landscape of two-layer neural networks},
 FJournal = {Proceedings of the National Academy of Sciences of the United States of America},
 Journal = {Proc. Natl. Acad. Sci. USA},
 ISSN = {0027-8424},
 Volume = {115},
 Number = {33},
 Pages = {e7665--e7671},
 Year = {2018},
 Language = {English},
 DOI = {10.1073/pnas.1806579115},
 zbMATH = {7076238},
 Zbl = {1416.92014}
}

@article{GuillinMonmarcheKinetic,
 Author = {Guillin, Arnaud and Monmarch{\'e}, Pierre},
 Title = {Uniform long-time and propagation of chaos estimates for mean field kinetic particles in non-convex landscapes},
 FJournal = {Journal of Statistical Physics},
 Journal = {J. Stat. Phys.},
 ISSN = {0022-4715},
 Volume = {185},
 Number = {2},
 Pages = {20},
 eid = {15},
 Year = {2021},
 Language = {English},
 DOI = {10.1007/s10955-021-02839-6},
 zbMATH = {7445218}
}

@article{DelarueTseUniformPOC,
      title={Uniform in time weak propagation of chaos on the torus},
      author={François Delarue and Alvin Tse},
      year={2024},
      eprint={2104.14973},
      archivePrefix={arXiv},
      primaryClass={math.AP},
      journal={Ann. Inst. Henri Poincaré, Probab. Stat.},
      pubstate={forthcoming}
}

@article{DEGZElementary,
 Author = {Durmus, Alain and Eberle, Andreas and Guillin, Arnaud and Zimmer, Raphael},
 Title = {An elementary approach to uniform in time propagation of chaos},
 FJournal = {Proceedings of the American Mathematical Society},
 Journal = {Proc. Am. Math. Soc.},
 ISSN = {0002-9939},
 Volume = {148},
 Number = {12},
 Pages = {5387--5398},
 Year = {2020},
 Language = {English},
 DOI = {10.1090/proc/14612},
 zbMATH = {7268404},
 Zbl = {1471.60123}
}

@article{DurmusMoulinesNonasymptotic,
 Author = {Durmus, Alain and Moulines, {\'E}ric},
 Title = {Nonasymptotic convergence analysis for the unadjusted {Langevin} algorithm},
 FJournal = {The Annals of Applied Probability},
 Journal = {Ann. Appl. Probab.},
 ISSN = {1050-5164},
 Volume = {27},
 Number = {3},
 Pages = {1551--1587},
 Year = {2017},
 Language = {English},
 DOI = {10.1214/16-AAP1238},
 zbMATH = {6775356},
 Zbl = {1377.65007}
}

@article{RobertsRosenthal,
 Author = {Roberts, Gareth O. and Rosenthal, Jeffrey S.},
 Title = {Optimal scaling of discrete approximations to {Langevin} diffusions},
 FJournal = {Journal of the Royal Statistical Society. Series B. Statistical Methodology},
 Journal = {J. R. Stat. Soc., Ser. B, Stat. Methodol.},
 ISSN = {1369-7412},
 Volume = {60},
 Number = {1},
 Pages = {255--268},
 Year = {1998},
 Language = {English},
 DOI = {10.1111/1467-9868.00123},
 zbMATH = {1212515},
 Zbl = {0913.60060}
}

@article{RobertsTweedie,
 Author = {Roberts, Gareth O. and Tweedie, Richard L.},
 Title = {Exponential convergence of {Langevin} distributions and their discrete approximations},
 FJournal = {Bernoulli},
 Journal = {Bernoulli},
 ISSN = {1350-7265},
 Volume = {2},
 Number = {4},
 Pages = {341--363},
 Year = {1996},
 Language = {English},
 DOI = {10.2307/3318418},
 zbMATH = {1005408},
 Zbl = {0870.60027}
}

@Article{BirkhoffErgodic,
 Author = {Birkhoff, George D.},
 Title = {Proof of the ergodic theorem},
 FJournal = {Proceedings of the National Academy of Sciences of the United States of America},
 Journal = {Proc. Natl. Acad. Sci. USA},
 ISSN = {0027-8424},
 Volume = {17},
 Pages = {656--660},
 Year = {1931},
 Language = {English},
 DOI = {10.1073/pnas.17.12.656},
 zbMATH = {3003496},
 Zbl = {0003.25602}
}

@article{CKRGame,
author = {Conforti, Giovanni and Kazeykina, Anna and Ren, Zhenjie},
title = {Game on Random Environment, Mean-Field {Langevin} System, and Neural Networks},
journal = {Math. Oper. Res.},
volume = {48},
number = {1},
pages = {78--99},
year = {2023},
doi = {10.1287/moor.2022.1252}
}

@article{ChizatMFL,
title={Mean-Field {Langevin} Dynamics: {Exponential} Convergence and Annealing},
author={L{\'e}na{\"i}c Chizat},
journal={Trans. Mach. Learn. Res.},
issn={2835-8856},
year={2022},
url={https://openreview.net/forum?id=BDqzLH1gEm},
}

@article{NWSPDA,
 Author = {Nitanda, Atsushi and Wu, Denny and Suzuki, Taiji},
 Title = {Particle dual averaging: optimization of mean field neural network with global convergence rate analysis},
 FJournal = {Journal of Statistical Mechanics: Theory and Experiment},
 Journal = {J. Stat. Mech. Theory Exp.},
 ISSN = {1742-5468},
 Volume = {2022},
 Number = {11},
 Pages = {51},
 eid = {114010},
 Year = {2022},
 Language = {English},
 DOI = {10.1088/1742-5468/ac98a8},
 zbMATH = {7632727}
}

@article{KuramotoOscillators,
title = {Rhythms and turbulence in populations of chemical oscillators},
journal = {Phys. A, Stat. Mech. Appl.},
volume = {106},
number = {1},
pages = {128--143},
year = {1981},
issn = {0378-4371},
doi = {10.1016/0378-4371(81)90214-4},
author = {Yoshiki Kuramoto},
}

@Article{MartzelAslangul,
 Author = {Martzel, Nicolas and Aslangul, Claude},
 Title = {Mean-field treatment of the many-body {Fokker}-{Planck} equation},
 FJournal = {Journal of Physics A: Mathematical and General},
 Journal = {J. Phys. A, Math. Gen.},
 ISSN = {0305-4470},
 Volume = {34},
 Number = {50},
 Pages = {11225--11240},
 Year = {2001},
 Language = {English},
 DOI = {10.1088/0305-4470/34/50/305},
 zbMATH = {1754002},
 Zbl = {0993.82036}
}

@Article{KellerSegelChemotaxis,
 Author = {Keller, Evelyn F. and Segel, Lee A.},
 Title = {Model for chemotaxis},
 FJournal = {Journal of Theoretical Biology},
 Journal = {J. Theor. Biol.},
 ISSN = {0022-5193},
 Volume = {30},
 Number = {2},
 Pages = {225--234},
 Year = {1971},
 Language = {English},
 DOI = {10.1016/0022-5193(71)90050-6},
 zbMATH = {5607960},
 Zbl = {1170.92307}
}

@Article{CMVKinetic,
 Author = {Carrillo, Jos{\'e} A. and McCann, Robert J. and Villani, C{\'e}dric},
 Title = {Kinetic equilibration rates for granular media and related equations: entropy dissipation and mass transportation estimates},
 FJournal = {Revista Matem{\'a}tica Iberoamericana},
 Journal = {Rev. Mat. Iberoam.},
 ISSN = {0213-2230},
 Volume = {19},
 Number = {3},
 Pages = {971--1018},
 Year = {2003},
 Language = {English},
 DOI = {10.4171/RMI/376},
 zbMATH = {2109873},
 Zbl = {1073.35127}
}

@Article{BGGGranular,
 Author = {Bolley, Fran{\c{c}}ois and Gentil, Ivan and Guillin, Arnaud},
 Title = {Uniform convergence to equilibrium for granular media},
 FJournal = {Archive for Rational Mechanics and Analysis},
 Journal = {Arch. Ration. Mech. Anal.},
 ISSN = {0003-9527},
 Volume = {208},
 Number = {2},
 Pages = {429--445},
 Year = {2013},
 Language = {English},
 DOI = {10.1007/s00205-012-0599-z},
 zbMATH = {6161066},
 Zbl = {1264.35040}
}

@Article{BCCPGranular,
 Author = {Benedetto, Dario and Caglioti, Emanuele and Carrillo, Jos{\'e} A. and Pulvirenti, Mario},
 Title = {A non-{Maxwellian} steady distribution for one-dimensional granular media},
 FJournal = {Journal of Statistical Physics},
 Journal = {J. Stat. Phys.},
 ISSN = {0022-4715},
 Volume = {91},
 Number = {5-6},
 Pages = {979--990},
 Year = {1998},
 Language = {English},
 DOI = {10.1023/A:1023032000560},
 zbMATH = {1279493},
 Zbl = {0921.60057}
}

@Article{DJLEmpApprox,
 Author = {Du, Kai and Jiang, Yifan and Li, Jinfeng},
 Title = {Empirical approximation to invariant measures for {McKean}--{Vlasov} processes: mean-field interaction vs self-interaction},
 FJournal = {Bernoulli},
 Journal = {Bernoulli},
 ISSN = {1350-7265},
 Volume = {29},
 Number = {3},
 Pages = {2492--2518},
 Year = {2023},
 Language = {English},
 DOI = {10.3150/22-BEJ1550},
 zbMATH = {7691590}
}

@Book{BKRSFPKEq,
 Author = {Bogachev, Vladimir I. and Krylov, Nicolai V. and R{\"o}ckner, Michael and Shaposhnikov, Stanislav V.},
 Title = {Fokker--{Planck}--{Kolmogorov} equations},
 FSeries = {Mathematical Surveys and Monographs},
 Series = {Math. Surv. Monogr.},
 ISSN = {0076-5376},
 Volume = {207},
 ISBN = {978-1-4704-2558-6},
 Year = {2015},
 Publisher = {American Mathematical Society},
 Address = {Providence, RI},
 Language = {English},
 DOI = {10.1090/surv/207},
 zbMATH = {6532864},
 Zbl = {1342.35002}
}

@Article{WangLandauType,
 Author = {Wang, Feng-Yu},
 Title = {Distribution dependent {SDEs} for {Landau} type equations},
 FJournal = {Stochastic Processes and their Applications},
 Journal = {Stochastic Process. Appl.},
 ISSN = {0304-4149},
 Volume = {128},
 Number = {2},
 Pages = {595--621},
 Year = {2018},
 Language = {English},
 DOI = {10.1016/j.spa.2017.05.006},
 zbMATH = {6824510},
 Zbl = {1380.60077}
}

@online{DJLSequentialPOC,
      title={Sequential propagation of chaos},
      author={Kai Du and Yifan Jiang and Xiaochen Li},
      year={2023},
      eprint={2301.09913},
      archivePrefix={arXiv},
      primaryClass={math.PR}
}

@online{DMRTVHighDimGaussians,
      title={The total variation distance between high-dimensional Gaussians with the same mean},
      author={Luc Devroye and Abbas Mehrabian and Tommy Reddad},
      year={2023},
      eprint={1810.08693},
      archivePrefix={arXiv},
      primaryClass={math.ST}
}

@article{FlemingViot,
 Author = {Fleming, Wendell H. and Viot, Michel},
 Title = {Some measure-valued {Markov} processes in population genetics theory},
 FJournal = {Indiana University Mathematics Journal},
 Journal = {Indiana Univ. Math. J.},
 ISSN = {0022-2518},
 Volume = {28},
 Pages = {817--843},
 Year = {1979},
 Language = {English},
 DOI = {10.1512/iumj.1979.28.28058},
 zbMATH = {3694306},
 Zbl = {0444.60064}
}

@article{BLRSIDiffusions,
 Author = {Bena{\"i}m, Michel and Ledoux, Michel and Raimond, Olivier},
 Title = {Self-interacting diffusions.},
 FJournal = {Probability Theory and Related Fields},
 Journal = {Probab. Theory Relat. Fields},
 ISSN = {0178-8051},
 Volume = {122},
 Number = {1},
 Pages = {1--41},
 Year = {2002},
 Language = {English},
 DOI = {10.1007/s004400100161},
 zbMATH = {1885537},
 Zbl = {1042.60060}
}

@article{CranstonLeJanSIDiffusions,
 Author = {Cranston, Michael and Le Jan, Yves},
 Title = {Self attracting diffusions: {Two} case studies},
 FJournal = {Mathematische Annalen},
 Journal = {Math. Ann.},
 ISSN = {0025-5831},
 Volume = {303},
 Number = {1},
 Pages = {87--93},
 Year = {1995},
 Language = {English},
 DOI = {10.1007/BF01460980},
 zbMATH = {840175},
 Zbl = {0838.60052}
}

@article{RaimondSIDiffusions,
 Author = {Raimond, Olivier},
 Title = {Self-attracting diffusions: {Case} of the constant interaction},
 FJournal = {Probability Theory and Related Fields},
 Journal = {Probab. Theory Relat. Fields},
 ISSN = {0178-8051},
 Volume = {107},
 Number = {2},
 Pages = {177--196},
 Year = {1997},
 Language = {English},
 DOI = {10.1007/s004400050082},
 zbMATH = {1002143},
 Zbl = {0881.60055}
}

@inproceedings{ChizatBachGlobal,
 author = {Chizat, L\'{e}na{\"i}c and Bach, Francis},
 booktitle = {Adv. Neural Inf. Process. Syst.},
 volume = {31},
 pages = {3036--3046},
 publisher = {Curran Associates, Inc.},
 title = {On the Global Convergence of Gradient Descent for Over-parameterized Models using Optimal Transport},
 url = {https://proceedings.neurips.cc/paper_files/paper/2018/file/a1afc58c6ca9540d057299ec3016d726-Paper.pdf},
 year = {2018}
}

@inproceedings{SNWUniform,
title={Uniform-in-time propagation of chaos for the mean-field gradient {Langevin} dynamics},
author={Taiji Suzuki and Atsushi Nitanda and Denny Wu},
booktitle={11th Int. Conf. Learn. Represent.},
year={2023},
url={https://openreview.net/forum?id=_JScUk9TBUn}
}

@Book{LedouxTalagrandProbaInBanach,
 Author = {Ledoux, Michel and Talagrand, Michel},
 Title = {Probability in {Banach} spaces. {Isoperimetry} and processes},
 FSeries = {Ergebnisse der Mathematik und ihrer Grenzgebiete. 3. Folge},
 Series = {Ergeb. Math. Grenzgeb., 3. Folge},
 ISSN = {0071-1136},
 Volume = {23},
 ISBN = {3-540-52013-9},
 Year = {1991},
 Publisher = {Springer-Verlag},
 Address = {Berlin etc.},
 Language = {English},
 zbMATH = {49190},
 Zbl = {0748.60004}
}

@InProceedings{NWSConvexMFL,
  title = 	 { Convex Analysis of the Mean Field {Langevin} Dynamics },
  author =       {Nitanda, Atsushi and Wu, Denny and Suzuki, Taiji},
  booktitle = 	 {Proc. 25th Int. Conf. Artif. Intell. Stat.},
  pages = 	 {9741--9757},
  volume = 	 {151},
  series = 	 {Proc. Mach. Learn. Res.},
  year = {2022},
  publisher =    {PMLR},
  url = 	 {https://proceedings.mlr.press/v151/nitanda22a.html}
}

@incollection{BakryEmeryHypercontractives,
 Author = {Bakry, Dominique and {\'E}mery, Michel},
 Title = {Diffusions hypercontractives},
 Year = {1985},
 Language = {French},
 booktitle = {S{\'e}min. de Probab. {XIX}},
 volume = {1123},
 series = {{Lect}. {Notes} {Math}.},
 publisher = {Springer-Verlag},
 address = {Berlin etc.}
}

@Article{BRSIDiffusions2,
 Author = {Bena{\"i}m, Michel and Raimond, Olivier},
 Title = {Self-interacting diffusions. {II}: {Convergence} in law.},
 FJournal = {Annales de l'Institut Henri Poincar{\'e}. Probabilit{\'e}s et Statistiques},
 Journal = {Ann. Inst. Henri Poincar{\'e}, Probab. Stat.},
 ISSN = {0246-0203},
 Volume = {39},
 Number = {6},
 Pages = {1043--1055},
 Year = {2003},
 Language = {English},
 zbMATH = {2003846},
 Zbl = {1064.60191}
}

@Article{BRSIDiffusions3,
 Author = {Bena{\"i}m, Michel and Raimond, Olivier},
 Title = {Self-interacting diffusions. {III}: {Symmetric} interactions},
 FJournal = {The Annals of Probability},
 Journal = {Ann. Probab.},
 ISSN = {0091-1798},
 Volume = {33},
 Number = {5},
 Pages = {1716--1759},
 Year = {2005},
 Language = {English},
 DOI = {10.1214/009117905000000251},
 zbMATH = {2228502},
 Zbl = {1085.60073}
}

@Article{BRSIDiffusions4,
 Author = {Bena{\"i}m, Michel and Raimond, Olivier},
 Title = {Self-interacting diffusions. {IV}: {Rate} of convergence},
 FJournal = {Electronic Journal of Probability},
 Journal = {Electron. J. Probab.},
 ISSN = {1083-6489},
 Volume = {16},
 Pages = {1815--1843},
 eid = {66},
 Year = {2011},
 Language = {English},
 DOI = {10.1214/EJP.v16-948},
 zbMATH = {6049122},
 Zbl = {1245.60089}
}

@Article{KKErgodicitySI,
 Author = {Kleptsyn, Victor and Kurtzmann, Aline},
 Title = {Ergodicity of self-attracting motion},
 FJournal = {Electronic Journal of Probability},
 Journal = {Electron. J. Probab.},
 ISSN = {1083-6489},
 Volume = {17},
 Pages = {37},
 eid = {50},
 Year = {2012},
 Language = {English},
 DOI = {10.1214/EJP.v17-2121},
 zbMATH = {6098133},
 Zbl = {1261.60093}
}
